\documentclass{article}
\usepackage{amsmath}
\usepackage{amssymb}
\usepackage{amsthm}
\usepackage{slashed}
\usepackage{cite}
\usepackage{color} 
\usepackage{caption}
\usepackage{euscript}
\usepackage{comment}
\usepackage[arrow,curve,matrix,arc,2cell]{xy}
\usepackage{graphicx}    
\usepackage[utf8]{inputenc}
\usepackage[unicode]{hyperref}
\usepackage[a4paper, margin=1.5in]{geometry}
\usepackage{pifont}
\UseAllTwocells
\DeclareFontFamily{U}{rsfs}{} \DeclareFontShape{U}{rsfs}{n}{it}{<->
rsfs10}{} \DeclareSymbolFont{mscr}{U}{rsfs}{n}{it}
\DeclareSymbolFontAlphabet{\scr}{mscr}
\def\mathscr{\scr}
\begin{document}

\newcommand{\red}[1]{\textcolor{red}{#1}}
\def\e#1\e{\begin{equation}#1\end{equation}}
\def\ea#1\ea{\begin{align}#1\end{align}}
\def\eas#1\eas{\begin{align*}#1\end{align*}}
\def\eq#1{{\rm(\ref{#1})}}
\newenvironment{claim}[1]{\par\noindent\underline{Claim:}\space#1}{}
\newenvironment{claimproof}[1]{\par\noindent\underline{Proof:}\space#1}{\hfill $\blacksquare$}
\theoremstyle{plain}
\newtheorem{thm}{Theorem}[section]
\newtheorem{prop}[thm]{Proposition}
\newtheorem{lem}[thm]{Lemma}
\newtheorem{cor}[thm]{Corollary}
\newtheorem{quest}[thm]{Question}
\newtheorem{conj}[thm]{Conjecture}
\theoremstyle{definition}
\newtheorem{dfn}[thm]{Definition}
\newtheorem{ex}[thm]{Example}
\newtheorem{rem}[thm]{Remark}
\numberwithin{equation}{section}

\def\loc{{\mathop{\rm loc}\nolimits}}
\def\ope{{\mathop{\rm op}\nolimits}}

\def\cyl{{\mathop{\rm cyl}\nolimits}}
\def\con{{\mathop{\rm con}\nolimits}}

\def\cay{\mathop{\rm cay}\nolimits}
\def\Cay{\mathop{\rm Cay}\nolimits}
\def\ACay{\mathop{\rm ACay}\nolimits}

\def\ACyl{{\mathop{\rm ACyl}\nolimits}}
\def\AC{{\mathop{\rm AC}\nolimits}}
\def\CS{{\mathop{\rm CS}\nolimits}}
\def\supp{\mathop{\rm supp}\nolimits}
\def\dist{\mathop{\rm dist}\nolimits}
\def\sgn{\mathop{\rm sgn}\nolimits}
\def\dim{\mathop{\rm dim}\nolimits}
\def\Ker{\mathop{\rm Ker}}
\def\Coker{\mathop{\rm Coker}}
\def\Ho{{\mathop{\rm H}}}
\def\sign{\mathop{\rm sign}\nolimits}
\def\id{\mathop{\rm id}\nolimits}
\def\dvol{\mathop{\rm dvol}\nolimits}
\def\spn{\mathop{\rm span}\nolimits}
\def\SO{\mathop{\rm SO}\nolimits}
\def\Orth{\mathop{\rm O}\nolimits}
\def\Fr{\mathop{\rm Fr}\nolimits}
\def\Gr{\mathop{\rm Gr}\nolimits}
\def\cay{\mathop{\rm cay}\nolimits}
\def\inj{\mathop{\rm inj}\nolimits}
\def\SF{\mathop{\rm SF}\nolimits}
\def\Or{\mathop{\rm Or}\nolimits}
\def\ad{\mathop{\rm ad}\nolimits}
\def\Hom{\mathop{\rm Hom}\nolimits}
\def\Map{\mathop{\rm Map}\nolimits}
\def\Crit{\mathop{\rm Crit}\nolimits}
\def\ev{\mathop{\rm ev}\nolimits}
\def\Univ{\mathop{\rm Univ}\nolimits}
\def\Fix{\mathop{\rm Fix}\nolimits}
\def\Hol{\mathop{\rm Hol}\nolimits}
\def\Iso{\mathop{\rm Iso}\nolimits}
\def\Hess{\mathop{\rm Hess}\nolimits}
\def\Stab{\mathop{\rm Stab}\nolimits}
\def\Pd{\mathop{\rm Pd}\nolimits}
\def\Aut{\mathop{\rm Aut}\nolimits}
\def\Diff{\mathop{\rm Diff}\nolimits}
\def\boFlag{\mathop{\rm Flag}\nolimits}
\def\boFlagSt{\mathop{\rm FlagSt}\nolimits}
\def\dOrb{{\mathop{\bf dOrb}}}
\def\dMan{{\mathop{\bf dMan}}}
\def\mKur{{\mathop{\bf mKur}}}
\def\Kur{{\mathop{\bf Kur}}}
\def\Re{\mathop{\rm Re}}
\def\Im{\mathop{\rm Im}}
\def\re{\mathop{\rm re}}
\def\im{\mathop{\rm im}}
\def\SU{\mathop{\rm SU}}
\def\Sp{\mathop{\rm Sp}}
\def\Spin{\mathop{\rm Spin}}
\def\GL{\mathop{\rm GL}}
\def\ind{\mathop{\rm ind}}
\def\area{\mathop{\rm area}}
\def\U{{\rm U}}
\def\vol{\mathop{\rm vol}\nolimits}
\def\virt{{\rm virt}}
\def\emb{{\rm emb}}
\def\bs{\boldsymbol}
\def\ge{\geqslant}
\def\le{\leqslant\nobreak}
\def\O{{\mathbin{\mathcal O}}}
\def\cA{{\mathbin{\mathcal A}}}
\def\cB{{\mathbin{\mathcal B}}}
\def\cC{{\mathbin{\mathcal C}}}
\def\cD{{\mathbin{\scr D}}}
\def\cDHS{{\mathbin{\scr D}_{\Q HS}}}
\def\cE{{\mathbin{\mathcal E}}}
\def\boE{{\mathbin{\mathbf E}}}
\def\cF{{\mathbin{\mathcal F}}}
\def\boF{{\mathbin{\mathbf F}}}
\def\cG{{\mathbin{\mathcal G}}}
\def\cH{{\mathbin{\mathcal H}}}
\def\cI{{\mathbin{\mathcal I}}}
\def\cJ{{\mathbin{\mathcal J}}}
\def\cK{{\mathbin{\mathcal K}}}
\def\cL{{\mathbin{\mathcal L}}}
\def\cM{{\mathbin{\mathcal M}}}
\def\bcM{{\mathbin{\bs{\mathcal M}}}}
\def\cN{{\mathbin{\mathcal N}}}
\def\cO{{\mathbin{\mathcal O}}}
\def\cP{{\mathbin{\mathcal P}}}
\def\boR{{\mathbin{\mathbf R}}}
\def\cS{{\mathbin{\mathcal S}}}
\def\cT{{\mathbin{\mathcal T}}}
\def\cU{{\mathbin{\mathcal U}}}
\def\cQ{{\mathbin{\mathcal Q}}}
\def\boQ{{\mathbin{\mathbf Q}}}
\def\cW{{\mathbin{\mathcal W}}}
\def\C{{\mathbin{\mathbb C}}}
\def\bQ{{\mathbin{\mathbb Q}}}
\def\bV{{\mathbin{\mathbb V}}}
\def\bE{{\mathbin{\mathbb E}}}
\def\bD{{\mathbin{\mathbb D}}}
\def\boF{{\mathbin{\mathbf F}}}
\def\bF{{\mathbin{\mathbb F}}}
\def\H{{\mathbin{\mathbb H}}}
\def\N{{\mathbin{\mathbb N}}}
\def\Q{{\mathbin{\mathbb Q}}}
\def\R{{\mathbin{\mathbb R}}}
\def\bS{{\mathbin{\mathbb S}}}
\def\Z{{\mathbin{\mathbb Z}}}
\def\sF{{\mathbin{\mathscr F}}}
\def\al{\alpha}
\def\be{\beta}
\def\ga{\gamma}
\def\de{\delta}
\def\io{\iota}
\def\ep{\epsilon}
\def\eps{\epsilon}
\def\la{\lambda}
\def\ka{\kappa}
\def\th{\theta}
\def\ze{\zeta}
\def\up{\upsilon}
\def\vp{\varphi}
\def\si{\sigma}
\def\om{\omega}
\def\De{\de}
\def\La{\Lambda}
\def\Si{\Sigma}
\def\Th{\Theta}
\def\Om{\Omega}
\def\Ga{\Gamma}
\def\Up{\Upsilon}
\def\pd{\partial}
\def\ts{\textstyle}
\def\st{\scriptstyle}
\def\sst{\scriptscriptstyle}
\def\w{\wedge}
\def\sm{\setminus}
\def\bu{\bullet}
\def\op{\oplus}
\def\ot{\otimes}
\def\ov{\overline}
\def\ul{\underline}
\def\bigop{\bigoplus}
\def\bigot{\bigotimes}
\def\iy{\infty}
\def\es{\emptyset}
\def\ra{\rightarrow}
\def\Ra{\Rightarrow}
\def\Longra{\Longrightarrow}
\def\ab{\allowbreak}
\def\longra{\longrightarrow}
\def\hookra{\hookrightarrow}
\def\dashra{\dashrightarrow}
\def\t{\times}
\def\ci{\circ}
\def\ti{\tilde}
\def\d{{\rm d}}
\def\dt{{\rm dt}}
\def\D{{\rm D}}
\def\Lie{{\mathcal{L}}}
\def\ha{{\ts\frac{1}{2}}}
\def\md#1{\vert #1 \vert}
\def\bmd#1{\big\vert #1 \big\vert}
\def\ms#1{\vert #1 \vert^2}
\def\nm#1{\Vert #1 \Vert}
\title{Conically singular Cayley submanifolds II: Desingularisations}
\author{Gilles Englebert}
\date{\today}
\maketitle

\begin{abstract} 
This is the second in a series of three papers working towards constructing fibrations of compact $\Spin(7)$ manifolds by Cayley submanifolds. In this paper we show that a conically singular Cayley submanifold in an almost $\Spin(7)$-manifold can be desingularised by gluing in a rescaled asymptotically conical submanifold of $\R^8$ with the same asymptotic cone. We assume that both the conically singular and the asymptotically conical Cayley submanifold are unobstructed in their respective moduli spaces.
\end{abstract}

\setcounter{tocdepth}{2}
\tableofcontents
\section{Introduction}
In the previous paper \cite{englebertDeformations}, we re-examined the deformation theory of compact Cayley submanifolds and extended it to conically singular as well as asymptotically conical Cayleys. 

The next step is the study of the desingularisation of conically singular Cayleys. Essentially this is asking the question of how a conically singular Cayley relates to nearby compact Cayleys. If $N$ is a conically singular submanifold of an ambient manifold $M$, and $A$ is asymptotically conical in $\R^8$ with the same cone, then one can hope to glue in a truncated version of $A$ at the singular point of $N$ to obtain a non-singular submanifold of $M$. Formally, this new submanifold can be thought of as a connected sum along their shared cone $N \sharp A$. This will however not usually be a calibrated submanifold, and thus another step is necessary. When $A$ is sufficiently close to the cone, and thus its non-conical regime is confined to a suitably small region, the glued manifold will be $C^k$-close to being calibrated. We can then perturb $N \sharp A$ to a nearby calibrated submanifold $\tilde{N}$. The analysis requires the deformation theory of both the asymptotically conical and the conically singular case, which we will revisit in this paper. For special Lagrangians in almost Calabi-Yau manifolds, the desingularisation has been described by Joyce \cite{joyceDesingularisation}. Lotay investigated the case of coassociative submanifolds in $G_2$-manifolds \cite{lotayDesingularizationCoassociative4folds2008}. The aim of this paper is to tackle the desingularisation of Cayleys in almost $\Spin(7)$-manifolds. To be precise, we will prove the following theorem, which also takes into account the presence of multiple singular points and partial desingularisation. 
\begin{thm}[Main theorem]
\label{1_1_Main_theorem}
Let $(M, \Phi)$ be an almost $\Spin(7)$ manifold and $N$ a $\CS_{\bar{\mu}}$-Cayley in $(M, \Phi)$ with singular points $\{z_i\}_{i = 0, \dots, l}$ and rates $1 < \mu_i < 2$, modelled on the cones $C_i = \R_+ \times L_i \subset \R^8$ . Assume that $N$ is unobstructed in $\cM_{\CS}^{\bar{\mu}} (N, \{\Phi\})$. For a fixed $k\le l$, assume for each $i \le k$ that the $L_i$ are unobstructed as associatives (i.e. that the $C_i$ are unobstructed cones), and that $\cD_{L_i} \cap (1, \mu_i] = \emptyset$. For $1 \le i \le k$, suppose that $A_i$ is an unobstructed $\AC_\la$-Cayley with $\la < 1$, such that $\cD_{L_i} \cap [\la, 1) = \emptyset$. Let $\{\Phi_s\}_{s \in \cS}$ be a smooth family of deformations of $\Phi = \Phi_{s_0}$. Then there are open neighbourhoods $U_i$ of $C_i \in  \overline{\cM}_{\AC}^{\la}(A_i)$, an open neighbourhood $s_0 \in U \subset \cS$  and a continuous map:
\e
\Ga: U \times \cM_{\CS, \text{cones}}^{\bar{\mu}} (N, \{\Phi\}) \times \prod_{i=1}^k U_i \longra \bigcup_{I \subset \{1, \dots, k\}} \cM_{\CS}^{\bar{\mu}_I}(N_I , \cS).
\e
Here we denote by ${\bar{\mu}_I}$ the subsequence, where we removed the $i$-th element for $i\in I$ from ${\bar{\mu}}$. Moreover, $N_I$ denotes the isotopy class of the manifold obtained after desingularising the points $z_i$ for $i \in I$ by a connected sum with $A_i$. 

This map is a local diffeomorphism of stratified manifolds. Thus away from the cones in $\overline{\cM}_{\AC}^{\la}(A_i)$ it is a local diffeomorphism onto the non-singular Cayley submanifolds. It maps the point $(s, \tilde{N}, \tilde{A}_1, \dots, \tilde{A}_k)$  into $\cM_{\CS}^{\bar{\mu}_I}(N_I , \cS)$, where $I$ is the collection of indices for which  $\tilde{A}_i = C_i$. This corresponds to partial desingularisation.
\end{thm}

We will follow a procedure similar to the ones in the two previously mentioned papers \cite{joyceDesingularisation} and \cite{lotayDesingularizationCoassociative4folds2008}, with two major differences. First, coassociative submanifolds and special Lagrangians tend to be unobstructed, and thus form smooth moduli spaces. This is not necessarily the case for Cayley submanifolds. We therefore need to restrict our results to unobstructed Cayleys, i.e. Cayleys which are have a surjective deformation operator, and are consequently smooth points in their moduli space. Secondly, it is not immediately clear what the generalisation of the deformation operator of a Cayley to a non-Cayley should be. We therefore carefully revisit the construction of the usual deformation operator and show how to extend it to 4-folds which are $C^1$-close to being Cayley, so called almost Cayleys. 

The motivation behind this theorem is the construction of fibrations of compact $\Spin(7)$-manifolds by Cayleys, in analogy to the programme by Kovalev in the $G_2$-case \cite{KovalevFibration}. A topological argument shows that a fibration of a compact $G_2$ manifold of full holonomy by compact coassociatives (the analogue of Cayleys in the $G_2$ setting) must necessarily include singular fibres, but there is hope that the singular fibres may admit only conical singularities. We do not currently have a proof of such a result in the $\Spin(7)$ case, however we expect a similar statement to hold. The desingularisation theory of our Cayley fibres and results such as our main Theorem \ref{1_1_Main_theorem} then allow us to investigate the deformation theory of fibrations of this type and deduce properties such as stability (in a suitable sense) under the deformation of the $\Spin(7)$-structure. 

\subsubsection*{Notation}
We will denote by $C$ an unspecified constant, which may refer to different constants within the same derivation. To indicate the dependence of this constant on other variables $x, y, \dots$, we will write it as $C(x, y, \dots)$. Similarly, if an inequality holds up to an unspecified constant, we will write $A \lesssim B$ instead of $A \le C B$.
We denote by $T^p_qM = (TM)^{\ot^p} \ot (T^*M)^{\ot^q} $ the bundle of $(p,q)$-tensors.

\subsubsection*{Acknowledgements}
This research has been supported by the Simons Collaboration on Special Holonomy in Geometry, Analysis, and Physics. I want to thank my DPhil supervisor Dominic Joyce for his excellent guidance during this project.

\section{Preliminaries}

To begin, we recall the definitions from the theory of $\Spin(7)$-manifolds and their Cayley submanifolds which are relevant to our discussion, as well as the deformation theory of compact and conically singular Cayley submanifolds.

\subsection{\texorpdfstring{Geometry of $\Spin(7)$-manifolds}{Geometry of Spin(7)-manifolds}}

For a more in depth discussion of $\Spin(7)$ and Cayley geometry we refer to \cite{HarvLaws} and \cite{mcleanDeformationsCalibratedSubmanifolds1998}. Recall that the group $\Spin(7)$ is the double cover of $\SO(7)$. It can alternatively be defined as the stabilizer of the \textbf{Cayley four form}:
\ea
\label{2_1_cayley_form}
\Phi_0 = \d x_{1234} - \d x_{1256} - \d x_{1278} - \d x_{1357} + \d x_{1368} - \d x_{1458} - \d x_{1467} \nonumber \\
- \d x_{2358} - \d x_{2367} + \d x_{2457} - \d x_{2468} - \d x_{3456} - \d x_{3478} + \d x_{5678}
\ea 
under the action of $\SO(8)$ on forms in $\La^4 \R^8$. We say that a pair $(M, \Phi)$, where $M$ is an $8$-dimensional manifold and $\Phi \in \Om^4 (M)$, is a $\Spin(7)$-manifold if at each point $p \in M$ there is an isomorphism $T_pM \simeq \R^8$ that takes $\Phi_p$ to $\Phi_0$. The Cayley form $\Phi$ induces a Riemannian metric $g_\Phi$ on $M$. If $\d \Phi = 0$ we say that the $\Spin(7)$-manifold is \textbf{torsion-free}. We then have that the holonomy $\Hol(g_\Phi)\subset \Spin(7)$. While determining whether an $8$-manifold is $\Spin(7)$ is purely topological in nature, finding a torsion-free $\Spin(7)$ structure is hard, comparable to finding integrable complex structures on almost complex manifolds.

Let $N^4 \subset M$ be any immersed submanifold. The Cayley form $\Phi$ then satisfies the Cayley inequality:
\e
\Phi|_N \le \dvol_{N}.
\e
Here $\dvol_N$ is the volume form induced by the metric $g_{\Phi}$. We say that a manifold $N$ is \textbf{Cayley} if $\Phi|_N = \dvol_N$, and that it is $\al$-\textbf{Cayley} for ($\al \in (0,1)$) if instead  $\Phi|_N \ge  \al\dvol_{N}$ (see also section 3 of \cite{englebertDeformations}). If we know that a submanifold is $\al$-Cayley for $\al$ close to $1$, we can hope that there is a true Cayley submanifold $C^1$-close to it. 

In $(\R^8, \Phi_0)$, the Cayley four planes form a $12$ dimensional subset of the $16$ dimensional Grassmannian of oriented four planes. Thus the Cayley condition can be seen as four independent equations. More precisely it can be described by the vanishing of a four-form $\tau \in \La^4 \R^8 \otimes E_{\cay}$, i.e. a plane $\pi = \spn\{e_1, e_2, e_3, e_4\}$ is Cayley exactly when $\tau(e_1, e_2, e_3, e_4) = 0$. Here the vector space $E_{\cay}$ can be thought of as the normal space of a Cayley as an element of the submanifold of Cayleys in $\Gr_+(4,8)$. From this form we get a deformation operator associated to an almost Cayley submanifold $N$. For $\al$ sufficiently close to $1$ we define:
\ea
\label{2_1_deformation_op}
\begin{array}{rl}
F: C^\infty (N,V) & \longrightarrow C^\infty(E_{\cay}) \\
v &\longmapsto \pi_E(\star_N \exp^*_v(\tau|_{N_v})) .
\end{array}
\ea
Here $V \subset \nu(N)$ is an open  neighbourhood of the zero section in the normal bundle of $N$. Moreover, $\exp_v: N \hookra M$ denotes the exponential of the vector field $v \in \nu(N)$. We denote the image of this map by $N_v$. The precise definition can be found in \cite[Section 3.1]{englebertDeformations}. Important for us is that this operator does indeed detect Cayley submanifolds and is elliptic at its zeroes, as in the following Proposition:

\begin{prop}[{Prop 3.3 and 3.5 in \cite{englebertDeformations}}]
\label{2_1_deformation_op_prop}
Let $(M, \Phi)$ be a $\Spin(7)$ manifold such that the Riemann tensor satisfies $\md{R} < C_R$. Let $\al < 1$ be sufficiently close to $1$. Then there is a constant $\eps > 0$ such that the following holds for any $\al$-Cayley $N$. If $v \in C^\infty(\nu(N))$ is such that for all $p \in N$:
\e
	\md{v(p)} < \frac{\ep}{\min\{\md{R(q)}: q \in B(p, \md{v(p)})\}},  \md{\nabla v(p)} < \eps,
\e
then $N_v$ is Cayley exactly when $F(v) = 0$. Furthermore for such $v$, the non-linear operator $F$ is elliptic at $v$.
\end{prop}

\subsection{Manifolds with ends}
We now briefly recall the definitions of asymptotically conical and conically singular manifolds. More details can be found in \cite[Section 2.3]{englebertDeformations}. Recall that an \textbf{asymptotically conical} manifold of rate $\eta < 1$  ($\AC_\eta$) is a Riemannian manifold $(M,g)$ such that away from a compact subset $K \subset M$ we can identify $M \setminus K \simeq (r_0, \infty) \times L$ and the metric satisfies: 
\e
\md{\nabla^i(g-g_{\con})} = O(r^{\eta- 1-i}) \text{ as } r \ra \infty,\label{2_1_ac_conv}
\e
where $g_{\con} = \d r^2 +r^2h$ is a conical metric on $(r_0, \infty) \times L$. For a embedded submanifold $f: A \hookra (\R^n, g)$, where $g$ asymptotically conical of rate $\eta < 1$ and asymptotic to flat $\R^n$, we say that it is an \textbf{asymptotically conical submanifold} of rate $\eta < \la < 1$  if: 
\ea
\md{\nabla^i(f(r,p)-\io(r,p))} \in O(r^{\la-i}), \text { as } r \ra \infty. \label{2_3_extrinsic_ac}
\ea
Here $\io: C \hookra \R^n$ is the embedding of the uniquely determined asymptotic cone of $A$. 

Finally, we say that a continuously embedded topological space $N \subset (M, \Phi)$ is \textbf{conically singular} with rates $\bar{\mu} = (\mu_1, \dots, \mu_l)$, where $1 <\mu_j<2$ ($\CS_{\bar{\mu}}$) asymptotic to cones $C_1, C_2, \cdots, C_l$ if it is a smoothly embedded manifold away from $l$ points $\{z_1, \dots, z_l\}$, and there are parametrisations $\Th_j: (0, R_0) \times L_j \ra M$ of $N$ near $z_j$ such that: 
\e
\md{\nabla^i(\Th_j(r,p)-\io_j(r,p))} \in O(r^{\mu_j-i}), \text { as } r \ra 0.
\e
Here $\io_j$ is the embedding of the cone $C_i$ via a parametrisation $\chi_j$ of $M$ that is compatible with the $\Spin(7)$-structure, i.e. $\D \chi_j^* \Phi(z_j) = \Phi_0$.  In both the $\AC$ and $\CS$ case, if a submanifold is Cayley, then it must be asymptotic to Cayley cones.

We now recall briefly the theory of Sobolev spaces as well as the Fredholm theory on conical manifolds. First of all, let $M$ be an $n$-dimensional manifolds with $l$ conical ends and $E$ a bundle of tensors on $M$. For a collection of weight $\bar{\de} \in \R^l$  and a section $s \in C^\infty_c(E)$ we define the Sobolev norm: 
\e
\nm{s}_{p, k, \bar{\de}} = \left(\sum_{i = 0}^k \int_{M} \md{\nabla^i s\rho^{-w+i}}^p \rho^{-n} \d \mu\right)^\frac{1}{p},
\e
Here $w$ is a weight function that interpolates between the conical ends. On the $j$-th conical end it is given by $\de_i$. We denote the completion of $C^\infty_c(E)$ under this norm by $L^p_{k, \bar{\de}}$. Note that the Banach space structure is independent of the choice of weight function. If there is only one end, we denote the space by $L^p_{k, \de}$. The $C^k_{\bar{\de}}$ spaces are defined as the completion with regards to the norm: 
\e
\nm{s}_{C^k_{\bar{\de}}} = \sum_{i = 0}^k \md{\nabla^i s\rho^{-w+i}}.
\e
There is a Sobolev embedding theorem: 

\begin{thm}[{\cite[Thm 4.8]{lockhartFredholmHodgeLiouville1987}}]
\label{2_3_Sobolev_embedding_acyl_ac}
Suppose that the following hold: 
\begin{itemize}
\item[i)] $k - \tilde{k} \ge n\left(\frac{1}{p}-\frac{1}{\tilde{p}}\right)$ and either:
\item[ii)] $ 1 < p \le \tilde{p} < \infty$ and $\tilde{\de} \ge \de$ \text{(}$\AC$\text{)} or $\tilde{\de} \le \de$ ($\CS$)
\item[ii')]$ 1 <  \tilde{p} < p < \infty$ and $\tilde{\de} > \de$ \text{(}$\AC$\text{)} or $\tilde{\de} < \de$ ($\CS$)
\end{itemize}
Then there is a continuous embedding: 
\e
L^p_{k, \de}(E) \longra L^{\tilde{p}}_{\tilde{k}, \tilde{\de}}(E) 
\e
\end{thm}
We call a linear differential operator $D$ that is asymptotically compatible with the conical structure a \textbf{conical operator}. This means that the coefficients of $D$ tend towards the coefficients of a rescaling invariant operator on the asymptotic cone. For a precise definition we refer to \cite[Section 2.3.3]{englebertDeformations}. It is of rate $\nu \in \R$ if $D [ C^{\infty}_{\de}] \subset C^{\infty}_{\de-\nu}$. Such operators define continuous operators between conical Banach spaces, and enjoy a good Fredholm theory.

\begin{thm}
\label{2_1_change_of_index}
Let $D$ be a conical operator of order $r\ge 0$ and rate $\nu \in \R$. Let $1<p<\infty$ and $k\ge 0$. Then it defines a continuous map: 
\eas
D: L^p_{k+r, \de}(E) \longra L^p_{k, \de-\nu}(E).
\eas
If $D$ is elliptic, then this map is Fredholm for $\de$ in the complement of a discrete subset $\cD \subset \R$. This subset is determined by en eigenvalue problem on the asymptotic link.

Denote by $i_\de(P)$ for $\de \in \R \setminus \cD$ the index of the operator $D: L^p_{k+r, \de}(E) \ra L^p_{k, \de-\nu}(E)$. We then have that: $i_{\de_2}(P)-i_{\de_2}(P) = \sum_{\la \in (\de_1, \de_2) \cap \cD} d(\la)$. Here $ d(\la)$ is determined solely by the asymptotic link. 
\end{thm}
\begin{ex}
\label{2_2_ex_weights}
By \cite[Ex 2.25]{englebertDeformations} we know that the critical weights of a Cayley plane seen as a cone or of an associative round $S^3 \subset S^7$ in the interval $(-4, 2)$ are given by: 
\e
 d(-3) = 1,\quad d(-1) = 1,\quad d(0) = 4, \quad d(1) = 12.
\e
Similarly, if we consider the quadratic complex cone $C_q = \{x^2 + y^2 + z^2 = 0, w = 0\} \subset \C^4$ with link $L \simeq \SU(2)/\Z_2$, the weight between $(-2, 2)$ are: 
\e
d(-1) = 2,\quad d(0) = 8, \quad d(1) = 22,\quad d(-1+\sqrt{5}) = 6.
\e
\end{ex}

Finally, there is a duality pairing between Sobolev spaces, where we need to change the weight as well.
\begin{prop}[cf. {\cite[Lem 2.8]{joyceReg}}]
\label{2_3_dual}
Assume that $1<p,q< \infty$ are such that $\frac{1}{p}+\frac{1}{q}= 1	$. Then if $n \in \N$ is the dimension of the underlying manifold and $\de \in \R^l$ is a vector of weights, there is a perfect pairing $L^p_{\de} \times L^q_{-n-\de} \longra \R$, and thus $(L^p_{\de})^* = L^q_{-n-\de}$.
\end{prop}

\subsection{Deformation theory of compact and conical Cayleys}

We now turn our attention to the results proven in \cite{englebertDeformations} about the deformation theory of compact, asymptotically conical and conically singular Cayley submanifolds. 

If $N \subset M$ is an immersed compact four dimensional submanifold in an $8$-manifold equipped with a smooth family of $\Spin(7)$-structures $\{\Phi_s\}_{s \in \cS}$, we consider the moduli space:
\ea
\label{3_3_cmpt_moduli}
\cM(N, \cS) = \{ (\tilde{N}, s): \tilde{N}& \text{ is an immersed Cayley submanifold of } (M, \Phi_s) \nonumber\\
 &\text{ with } \tilde{N} \text{ isotopic to } N\}. 
\ea
When the family $\cS$ has a single element $\Phi$ we also this moduli space by $\cM(N, \Phi)$. The structure of this moduli space has been studied in the foundational paper \cite{mcleanDeformationsCalibratedSubmanifolds1998} and in \cite{mooreDeformationTheoryCayley2017} in the case of torsion-free $\Spin(7)$-structures. We extended this result in \cite{englebertDeformations} to include families of $\Spin(7)$-structure with torsion, as in the following theorem.

\begin{thm}[{\cite[Thm. 4.9] {englebertDeformations}}]
\label{2_3_compact_structure}
There is a family of non-linear deformation operators $F_s$ ($s \in \cS$) for which for $\ep> 0$ sufficiently small give a smooth map: 
\eas
F: \mathcal{L}_\eps = \{ v\in L^p_{k+1} (\nu_\ep (N)), \nm{v}_{L^p_{k+1}} < \ep \} \times \cS \longrightarrow L^{p}_{k}(E).
\eas
A neighbourhood of $(N, s_0)$ in $\cM(N, \cS)$ is homeomorphic to the zero locus of $F$ near $(0, s_0)$. Furthermore we can define the \textbf{deformation space} $\cI(N, \cS) \subset C^\infty (\nu (N))$ to be the the kernel of $D_{N, s_0} = \D F (0, s_0)$, and the \textbf{obstruction space} $\cO(N, \cS) \subset C^\infty(E_{\cay})$ to be the cokernel of $D_{N, s_0}$. Then a neighbourhood of $(N, s_0)$ in $\cM(N, \cS)$ is also homeomorphic to the zero locus of a smooth Kuranishi map: 
\eas
	K:  \cI(N, \cS) \longra \cO(N, \cS).
\eas
In particular if $\cO(N, \cS) = \{0\}$ is trivial, $\cM(N, \cS) $ admits the structure of a smooth manifold near $(N, s_0)$. We say that $N$ is \textbf{unobstructed} in this case.
\end{thm}
Note that the notion of unobstructedness depends on the family $\cS$. The same Cayley submanifold $N \subset (M,\Phi)$ may be unobstructed in some $\cM(N, \cS)$ while being obstructed in $\cM(M,\Phi)$. In the embedded case we can express the index of the linearised operator (thus also the virtual dimension of $\cM(N, \cS)$ by the formula: 
\e
\label{2_3_form_index}
\ind D_{N, s_0} = \frac{1}{2}(\si(N) + \chi(N)) - [N] \cdot [N] + \dim \cS.
\e
Here $\si(N)$ and $\chi(N)$ are the signature and Euler characteristic of $N$ respectively, and $[N] \cdot [N]$ is the self-intersection number of $N$ in $M$.

Next, we turn our attention to the conical theory, starting with the asymptotically conical case. So let $A \subset \R^8$ be an $\AC_\la$ Cayley for a $\Spin(7)$-structure $\Phi$ that is $\AC_\eta$ to the flat $\Phi_0$. Here we require that $\eta < \la < 1$. Assume that $A$ is asymptotic to a single cone $C=\R_{>0} \times L$. If $\{\Phi_s\}_{s\in\cS}$ is a family of smooth $\AC_\eta$ deformations of $\Phi$, then we consider the moduli space: 
\ea
\label{2_3_ac_moduli}
\cM_\AC^\la(A, \mathcal{S}) = \{ (\tilde{A}, \Phi_s): \tilde{A}& \text{ is an } \AC_\la \text{ Cayley submanifold of } (\R^8, \Phi_s) \nonumber\\
 &\text{ isotopic to } A\text{ and asymptotic to the same cone} \}. 
\ea
It is important to keep track of the rate $\la$ of the Cayley as the expected dimension of the moduli space above as well as unobstructedness will depend on it. The following is the structure theorem for these moduli spaces:

\begin{thm}[{\cite[Thm. 4.18]{englebertDeformations}}]
\label{2_3_ac_structure}
There is a non-linear deformation operator $F_{\AC}$ which for $\ep> 0$ sufficiently small is a $C^\infty$ map: 
\eas
F_{\AC}: \mathcal{L}_\eps = \{ v\in L^p_{k+1, \la} (\nu_\ep (A)), \nm{v}_{L^p_{k+1, \la}} < \ep \} \times \cS \longrightarrow L^{p}_{k, \la-1}(E).
\eas
A neighbourhood of $(A, \Phi)$ in $\cM^\la_{\AC}(A, \cS)$ is homeomorphic to the zero locus of $F_{\AC}$ near $0$. If $\la \not \in \cD_L$ is not critical then the linearised operator $D_{\AC} = \D F_{\AC}(0,s_0)$ is Fredholm. We define the \textbf{deformation space} $\cI^{\la}_{\AC}(A) \subset C^\infty_{\la} (\nu (A))\times T_\Phi \cS$ to be the the kernel of $D_{\AC}$, and the \textbf{obstruction space} $\cO^{\la}_{\AC}(A) \subset C^\infty_{4-\la}(E)$ to be the cokernel of $D_{\AC}$. Both are finite dimensional. Then a neighbourhood of $A$ in $\cM^\la_{\AC}(A, \cS)$ is also homeomorphic to the zero locus of a smooth Kuranishi map: 
\eas
	K^\la_{\AC}:  \cI^\la_{\AC}(A) \longra \cO^\la_{\AC}(A).
\eas
In particular if $\cO^\la_{\AC}(A) = \{0\}$ is trivial, $\cM^\la_{\AC}(A, \cS) $ admits the structure of a smooth manifold near $A$. We say that $A$ is \textbf{unobstructed} in this case.
\end{thm}

Later we will need a bound on the inverse of $D_{\AC}$ on the complement of its kernel, when $A$ is an $\al$-Cayley. 

\begin{prop}
\label{2_3_ac_pseudo}
Suppose that $A \subset (\R^8, \Phi_0)$ is $\AC_\la$ to a Cayley cone with $\la <  1$ and $\al$-Cayley for $\al$ sufficiently close to $1$. Let $\de \in \R$ with $\de\not\in \cD_L$ and suppose $p > 4$, $k \ge 1$, $\eps > 0$ small. Then there is a subspace $\ka_\AC \subset C^\infty_{c}(\nu(A))$ such that for any $v \in \ker D_\AC \subset L^p_{k+1,\de -\eps}(\nu_\ep(A))$ we have that if $v$ is $L^2_{\de-\eps}$-orthogonal to $\ka_\AC$, then $v$ must vanish. This subspace, called a \textbf{pseudo-kernel}, can be chosen of the same dimension as $\ker D_{\AC}$. If we identify the normal bundles of $A$ for small $\AC_\eta$ perturbations of the $\Spin(7)$-structure via orthogonal projection, then $\ka_{\AC}$ is also a pseudo-kernel for small perturbations.
\end{prop}
\begin{proof}
As the operator $D_{\AC}$ is Fredholm by assumption, we know that $\ker D_{\AC}$ is finite-dimensional. Now by \cite[Cor. 4.5]{lockhartFredholmHodgeLiouville1987} we can approximate a given basis $\{a_i\}_{i=1}^l$ of $\ker D_{\AC}$ arbitrarily well in $L^p_{k+1,\de }$  by $C^\infty_{c}$ sections. By the Sobolev embedding $L^p_{k+1,\de} \hookrightarrow L^2_{0, \de -\eps}$ the same is true for $L^2_{\de -\eps}$. These approximations gives us the desired subspace $\kappa_{\AC}$. For nearby $\Spin(7)$-structures this result remains true, as the kernel is perturbed continuously in $L^p_{k+1, \de-\eps }$ by $\AC_\eta$ perturbations of the ambient $\Spin(7)$-structure.
\end{proof}

\begin{prop}
\label{2_3_ac_bound_k}
In the situation of Proposition \ref{2_3_ac_pseudo} there is a constant $C_{\AC}$ such that the following holds. If $v\in L^p_{k+1, \de}(\nu(A))$ is $L^2_{\de-\eps}$-orthogonal to $\ka_{\AC}$ then: 
\e
\nm{v}_{L^p_{k+1, \de}} \le C_\AC \nm{D_{\AC} v}_{L^p_{k, \de-1}}. 
\e
The same inequality holds true for small $\AC_\eta$ perturbations of the $\Spin(7)$-structure.
\end{prop} 
\begin{proof}
The map $D_\AC: L^p_{k+1, \de}  \longra L^{p}_{k, \de-1}$ is continuous by Proposition \ref{2_3_ac_structure} and has finite-dimensional co-kernel by the assumption on the weight $\de$. We claim that $\tilde{D} = D_{\AC}|_{\kappa_{\AC}^\perp}$ is an isomorphism onto its image, where the the orthogonal complement is taken with respect to the $L^2_{\de-\eps}$ inner product. Indeed it is injective by the construction of $\kappa_{\AC}$.  Moreover if $w \in \im \tilde{D}$, then we can find a pre-image $v \in L^p_{k+1, \de }(\nu(A))$ of $w$ as follows. Since $L^p_{k+1, \de }(\nu(A)) = \ka_{\AC}^\perp \op \ker D_\AC$, we can consider the $\ka_\AC^\perp$-component $v'$ of $v$, and note that $D_\AC v' = w$, thus proving surjectivity. Since $\tilde{D}$ is bijective and continuous, it admits a bounded inverse by the open mapping theorem for Banach spaces.
\end{proof}

Finally we discuss the conically singular case. The treatment is very similar, except that there may be more than one singular point. Let $N \subset (M, \Phi)$ be an $\CS_{\bar{\mu}}$ Cayley (having singular points $\{z_i\}_{1\le i \le l}$ asymptotic to the cones $C_i$ with rates $1< \mu_i <2$) for a $\Spin(7)$-structure $\Phi$. If $\{\Phi_s\}_{s\in\cS}$ is a family of smooth $\CS_\eta$ deformations of $\Phi$, then we consider the moduli space with moving points and cones: 

\begin{align*}
\cM_{\CS}^{\bar{\mu}}(N, \cS) = \{ (\tilde{N}, s): & \tilde{N} \subset (M, \Phi_s) \text{ is a } \CS_{\bar{\mu}}\text{-Cayley with singularities } \tilde{z}_1, \dots, \tilde{z}_l \\ &\text{ and cones } \tilde{C}_1, \dots, \tilde{C}_l, 
\text{. Here } \tilde{N} \text{ is isotopic to } N \text{, where }\\ &\text{ the isotopy takes } z_i \text{ to } \tilde{z}_i   \text{, and  } \tilde{C}_i \text{ is a deformation of } C_i \}.
\end{align*}

One can think about this moduli space as including deformations of rate $\mu = 1$ (rotations and Cayley deformations of the cones) and $\mu = 0$ (translations) manually. We have a structure theorem that is analogous to the asymptotically conical version \ref{2_3_ac_structure}.

\begin{thm}[{\cite[Thm. 4.32]{englebertDeformations}}]
\label{2_3_cs_structure}
There is a non-linear deformation operator $F_{\CS}$ which for $\ep> 0$ sufficiently small gives a $C^\infty$ map: 
\eas
F_{\CS}: \mathcal{L}_\eps = \{ v\in L^p_{k+1, {\bar{\mu}}} (\nu_\ep (N)), \nm{v}_{L^p_{k+1, {\bar{\mu}}}} < \ep \} \times \cC \longrightarrow L^{p}_{k, {\bar{\mu}-1}}(E).
\eas
Here $\cO$ is a smooth manifold that parametrises the deformations of the cones, the movement of the singular points, and the change in $\Spin(7)$-structure. Then a neighbourhood of $N$ in $\cM^{\bar{\mu}}_{\CS}(N, \cS)$ is homeomorphic to the zero locus of $F_{\CS}$ near $0$. Furthermore, if $\mu_i \not \in \cD_{L_i}$  for all $i$, the linearised operator is Fredholm and  we can define the \textbf{deformation space} $\cI^{\bar{\mu}}_{\CS}(N) \subset C^\infty_{\bar{\mu}} (\nu (N))\op T_{f_0}\cO$ to be the the kernel of $D_{\CS} = \D F_{\CS}(0)$, and the \textbf{obstruction space} $\cO^{\bar{\mu}}_{\CS}(N) \subset C^\infty_{4-\bar{\mu}}(E)$ to be the cokernel of $D_{\CS}$. Then a neighbourhood of $N$ in $\cM^{\bar{\mu}}_{\CS}(N, \cS)$ is also homeomorphic to the zero locus of a Kuranishi map: 
\eas
	K^{\bar{\mu}}_{\CS}:  \cI^{\bar{\mu}}_{\CS}(N) \longra \cO^{\bar{\mu}}_{\CS}(N).
\eas
In particular if $\cO^{\bar{\mu}}_{\CS}(N) = \{0\}$ is trivial, $\cM^{\bar{\mu}}_{\CS}(N, \cS) $ admits the structure of a smooth manifold near $N$. We say that $N$ is \textbf{unobstructed} in this case.
\end{thm}

\begin{rem}
We can similarly define a non-linear deformation operator $F_{\CS, \text{cones}}$, where we do not allow the conically singular points to move. This can be done by restricting $\cO$ to only include deformations of the cone that keep the vertices fixed. We obtain a moduli space $\cM^{\bar{\mu}}_{\CS, \text{cones}}(N, \cS)$ of Cayleys which have the same conically singular points as $N$. Finally, we can restrict the moduli space even further to $\cM^{\bar{\mu}}_{\CS, \text{fix}}(N, \cS)$, where neither the points nor the cones are allowed to move. This corresponds to not adding in additional deformations of the cone into $\cO$, and we call the corresponding operator $F_{\CS, \text{fix}}$.
\end{rem}

We can now define a notion of pseudo-kernel as in \ref{2_3_ac_pseudo}. This is entirely analogous, except that the Sobolev embedding $L^2_{\de + \eps} \ra L^p_{k+1, \de}$ requires us to slightly increase the rate of the $L^2$ sections.

\begin{prop}
\label{2_3_cs_pseudo}
Suppose that $N$ is $\CS_{\bar{\mu}}$ to Cayley cones and $\alpha$-Cayley for $\al$ sufficiently close to $1$. Let $\de \in \R$ with $\de \not\in \cD_{L_i} $ not critical for any of the links of $N$ and suppose $p > 4$, $k\ge 1$ and $ \eps > 0$ small. Then here is a subspace $\ka_{\CS} \subset C^\infty_{c}(\nu(N))$  such that for any $v \in \ker D_{\CS} \subset L^p_{k+1,\de}(\nu(N))$ we have that if $v$ is $L^2_{\de +\eps}$-orthogonal to $\ka_{\CS}$, then $v$ must vanish. This subspace, called a \textbf{pseudo-kernel} can be chosen of dimension $dim\ker D_{\CS}$. 
\end{prop}

\begin{prop}
\label{2_3_cs_bound_k}
In the situation of Proposition \ref{2_3_cs_pseudo} there is a constant $C_{\CS}$ such that the following holds. If $v\in L^p_{k+1, \bar{\de}}(\nu(N))$ is $L^2_{\bar{\de} + \eps}$-orthogonal to $\ka_\CS$ then: 
\e
\nm{v}_{L^p_{k+1, {\bar{\de}}}} \le C_{\CS} \nm{D_{\CS} v}_{L^p_{k, {\bar{\de}-1}}}. 
\e
The same inequality is true for perturbations of $N$ with $\bar{\mu} \ge \bar{\de}$.
\end{prop} 

The operator $F_{\CS}$ allows for the points of the singular cones to move. We could also fix the points while still allowing the links of the cones to deform, giving us an operator $F_{\CS, 0}$. We can give this operator the exact same treatment and reprove all the theorems in this section.

\section{Desingularisation of conically singular Cayleys}

We will now tackle the proof of the main Theorem \ref{1_1_Main_theorem}. We first describe a gluing construction which will yield an approximate Cayley desingularisation. Afterwards we describe an iteration scheme which will allows us to perturb the approximate Cayley to a true Cayley. We modify the construction from \cite{lotayDesingularizationCoassociative4folds2008} to work in families, and rework some analytic aspects to remove the requirements on the rate $\la$ of the asymptotically conical pieces.

\subsection{Approximate Cayley}

Let $(M, \Phi)$ be an almost $\Spin(7)$-manifold and let $\{\Phi_s\}_{s \in \cS}$ be a smooth family of deformations of $\Phi = \Phi_{s_0}$. Suppose $N$ is an unobstructed $\CS_{\bar{\mu}}$-Cayley in $(M, \Phi)$ with singular points $\{z_i\}_{i = 1, \dots, l}$. Let $\bar{\mu}$ be such that $(1, \mu_i] \cap \cD_{L_i} =  \emptyset$. Note that if the locus of singular points (which is smooth by unobstructedness) were to move by an ambient isotopy $I_s$, we can choose a new family $\{I_s^*\Phi_s\}_{s \in \cS}$ that leaves the singular locus invariant. Furthermore, we can also assume that $\Phi_s(z_i) = \Phi_{s_0}(z_i)$. For $B_\eta(0)$ the ball of radius $\eta > 0$ in $\R^8$, let $\chi_i : B_\eta(0) \ra M$ be a $\Spin(7)$-coordinate system centred around $z_i$. Recall that this means that $\chi_i$ is a parametrisation of a neighbourhood of $z_i$, such that $\chi_i(0) = z_i$ and $\D \chi_i|_0^* \Phi_{z_i} = \Phi_0$. After identifying $T_{z_i}M$ with $\R^8$ via the $\Spin(7)$-isomorphism $\D \chi_i|_0$, we let $(L_i, h_i) \subset (S^7, g_{\mathrm{round}})$ be the link on which the conical singularity is modelled. Assume it comes in a smooth, finite dimensional moduli space $\cM^{G_2}(L_i)$. For $1 \le i \le k \le l$ and  let $A_i$ be $\AC_\la$ Cayleys in $\R^8$ with the standard $\Spin(7)$-structure, with $\la < 1$ and $\cD_{L_i} \cap (\la, 1)= \emptyset$. Let the link of $A_i$ be $(L_i, h_i)$, and suppose we have chosen a scale function $t_i: \cM^\la_{\AC}(A_i) \ra \R$. We will now describe a procedure which allows us to glue elements of sufficiently small scale in $\overline{\cM}^\la_{\AC}(A_i)$ onto the first $k$ singular points of $N \in \cM^{\la}_\CS (N_0,\Phi_s)$, to produce Cayleys in $(M, \Phi_s)$ that are close to being singular. Here we need to make sure to glue compatible cones, as both moduli spaces allow for deformations of the cone. So assume that $A_i$ has a cone that is compatible with the singularity at $z_i$. In the gluing construction, the scale $t_i$ determines both the scaling of the $\AC$ piece $A_i$ as well as the inner radius of the annuli joining $A_i$ to $N$, which is comparable to $L_i \times (t_i r_0, R_0)$. In particular, when $t_i = 0$  (which corresponds to the cone in $\overline{\cM}^\la_{\AC}(A_i)$) we do not glue anything into the singularity at $z_i$. Recall that from the definition of a conically singular submanifold there is a compact set $K_N \subset N$ and decomposition $N = K_N \bigsqcup_{i = 1}^l U_i$ such that we have diffeomorphisms $\Psi_{\CS}^i : L_i \times(0, R_0) \ra U_i$. Choose $\eta$ and $R_0$ in such a way that the image of $\Psi_\CS^i$ is contained in $\chi_i(B_\eta(0))$. We can then factor $\Psi_\CS^i = \chi_i \circ \Th_\CS^i$, where $\Th_\CS^i$ is a smooth map $\Th_\CS^i: L_i \times (0, R_0) \ra B_\eta(0)$. For $1 \le i \le k$ there is a similar diffeomorphism $\Th_\AC^i: L_i \times (r_0, \infty) \ra A \setminus K_{A_i} \subset \R^8$, where $K_{A_i}$ is a compact subset of $A_i$, which can be chosen such that $\nm{\Th_\AC^i(p)-\iota_{i}(p)}_{\R^8} = O(\md{p}^{\la+1})$ as $p \ra \infty$. After reducing the scale of the $A_i$ we consider, we can assume that $r_0 < R_0$ and $A_i \setminus \Th^i_\AC(L_i \times (R_0, \infty))$ is contained in $B_\eta(0)$. In particular we can then also consider the map $\Psi_\AC^i|_{A_i \setminus L_i \times (R_0, \infty)} = \chi_i \circ \Th_\AC^i$. Now fix a smooth cut-off function $\phi: \R \ra [0, 1]$ with the property that:
\e
\label{4.1_cutoff}
\phi|_{(-\infty, \frac{1}{4}]} = 0, \quad \phi|_{[\frac{3}{4}, +\infty) } = 1.
\e

Let a constant $0 < \nu < 1$ be given. Assume $t > 0$ to be sufficiently small, so that we have the inequalities $0 < r_0t < \frac{1}{2} t^\nu < t^\nu < R_0 < 1$. Suppose that $\bar{A} = (A_1, \dots, A_k)$ is a collection of $\AC_\la$ manifolds (or cones) as above such that $t_i = t_i(A_i) \le t$. In this case we call $t$ the \textbf{global scale} of $\bar{A}$. We then define the subsets $N^{\bar{A}}$ of $M$ as follows: 
\ea
	\label{4.1_tentative_cayley}
	N^{\bar{A}} &= \bigg(N \setminus \bigsqcup_{i = 1}^l \Psi^i_\CS(L_i\times (0, t_i^\nu))\bigg) \sqcup \bigsqcup_{i = 1}^k \Psi^i_{\bar{A}}(L_i \times (r_0t_i, t_i^\nu)) \nonumber \\
	&\sqcup \bigsqcup_{i = 1}^k \chi_i(A_i \setminus \Th^i_\AC ( r_0t_i, \infty))\sqcup \bigsqcup_{i = k+1}^l \Psi^i_\CS(L_i\times (0, t_i^\nu)).
\ea

Here $\Th^i_{\bar{A}}$ is defined as the following interpolation between the $A_i$ and $U_i$ pieces:
\ea
 \Th^i_{\bar{A}}: &L_i \times (r_0t_i, R_0) \longra  \R^8 \nonumber \\
 		(p, s) &\longmapsto (1-\phi)\left(\frac{2s}{t_i^\nu}-1\right) \Th^i_\AC(p, s) + \phi\left(\frac{2s}{t_i^\nu}-1\right) \Th^i_\CS(p, s).
\ea
If we reduce the scale of a subset of the asymptotically conical pieces, the resulting family are desingularisations of $N$ where some tips shrink back to conically singular points. In particular, if $t_i = 0$, we should interpret the above definition as $\Th_{\bar{A}}^i = \Th_{\CS}^i$, thus the corresponding singularity is left as is, without gluing. As before, we also have maps $\Psi^i_{\bar{A}} = \chi_i \circ \Th^i_{\bar{A}}$, so we can work in local coordinates around a singularity. Notice that as $\phi$ is locally constant in neighbourhoods of $0$ and $1$, the $N^{\bar{A}}$ is in fact a smooth submanifold. For analytic purposes we usually consider $N^{\bar{A}}$ as a union of four parts: 
\begin{itemize}
\item[\ding{192}] $N^{\bar{A}}_u = \left(N \setminus \bigsqcup_{i = 1}^k \Psi^i_\CS\left(L_i\times \left(0, t_i^\nu\right)\right)\right)$.
\item[\ding{193}]  $N^{\bar{A}}_m = \bigsqcup_{i = 1}^k \Psi^i_{\bar{A}}(L_i \times (r_0t_i, t_i^\nu)) =\bigsqcup_{i = 1}^k  N^{A_i}_m$.
\item[\ding{194}]  $N^{\bar{A}}_l = \bigsqcup_{i = 1}^k \chi_i\left(A_i \setminus \Th^i_\AC \left(r_0t_i, \infty\right)\right)=\bigsqcup_{i = 1}^k  N^{A_i}_l $.
\item[\ding{195}]  $N^{\bar{A}}_{p} = \bigsqcup_{i = k+1}^l \bigsqcup_{i = 1}^l \Psi^i_\CS(L_i\times (0, t^\nu))$.
\end{itemize}

\begin{figure}
\begin{center}
\includegraphics[scale=0.7]{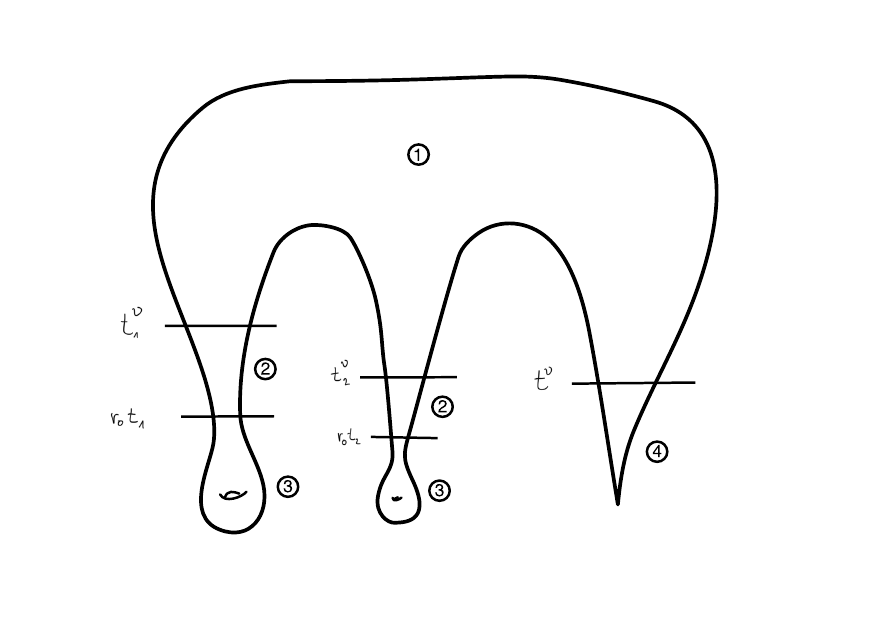}
\end{center}
\caption{Glued manifold}
\label{3_4_gluing}
\end{figure}

Notice that as we chose our family $\cS$ in such a way as to leave the singular locus unchanged, as well as the $\Spin(7)$-structure at the singular points, we can use the same frames $\chi_i$ for all deformations of the $\Spin(7)$-structure, and all nearby gluing data with matching cones. 

The reason for making the lower part shrink sub-linearly while the  tip shrinks linearly, is that $N^{\bar{A}}_l$ will stretch out and approximate any compact subset of the $A_i$ arbitrarily well as the scale is reduced. We now show that indeed this construction results in an approximation that is $C_1$-close to a Cayley in the following sense:

\begin{prop}
\label{4.1_alpha_cayley}
Let $\al \in (-1, 1)$ be given. Then if the global scale $t$ is sufficiently small, $N^{\bar{A}}$ is $\al$-Cayley.
\end{prop}
\begin{proof}
It is clear that $N^{\bar{A}}_u$ and $N^{\bar{A}}_{p}$ are always $\al$-Cayley, since they are subsets of $N$ and thus Cayley for the $\Spin(7)$-structure $\Phi_s$. Now for $N^{\bar{A}}_m$ and $N^{\bar{A}}_l$, we note that for $x\in \R^8$ near $0$ we have $(\D \chi_i)_x^* (\Phi_{\chi_i(x)}) = \Phi_0 + O(\nm{x})$. Thus for $t$ sufficiently small, we have for any $p \in N^{\bar{A}}_m \cup N^{\bar{A}}_l$ that $(\D \chi_i)_{\chi_i^{-1}(p)}^* (\Phi_p) = \Phi_0 + O(t^\nu)$. As $A_i$ is already Cayley for $\Phi_0$, it will also be $\al$-Cayley for $(\D \chi_i)_{\chi_i^{-1}(p)}^* (\Phi_p)$ for sufficiently small values of $t$. 

It remains to show that $N^{\bar{A}}_m$ is $\al$-Cayley for $t$ sufficiently small. Now by assumption on $N$ and the $A_i$, $\Th_\CS^i(p, s)$ and $t_i\Th_\AC^i(p, s)$ approach the same Cayley cone as long as $s \in (\frac{1}{2} t^\nu, t^\nu)$ and $t_i \longra 0	$, and thus the respective tangent planes become arbitrarily close to the same Cayley plane, in particular they will be $\al'$-Cayley for $t$ small enough and any $\al' > \al$. Now for every $\al$ there is an $\al' > \al$ such that if $\xi_1, \xi_2$ are $\al'$-Cayley graphs over a Cayley $\xi$, any linear interpolation of the between the maps having image $\xi_1$ and $\xi_2$ will have image an $\al$-Cayley. Thus $N^{\bar{A}}_m$ will also be $\al$-Cayley for $t$ small enough.
\end{proof}

Our goal is to construct Cayley submanifolds close to the almost Cayley submanifolds $N^{\bar{A}}$. To simplify the analytic details, we will introduce Banach spaces tailored to this particular desingularisation. These were first defined in \cite{lotayDesingularizationCoassociative4folds2008}. Before that, we extend our notion of a radius function to the $N^{\bar{A}}$,  combining the definitions of radius functions on $\CS$- and $\AC$-manifolds.
\begin{dfn}
A collection of \textbf{radius functions} on $N^{\bar{A}}$ for all $\bar{A}$ with global scale bounded by $t > 0$ is a smooth function $\rho: N^{\bar{A}} \ra [0, 1]$ such that:  
\e
\rho(x) = \left\{ 
\begin{array}{cc}
\Th(R_0), & x\in K_N \\
\Th(r_0t_i), & 1 \le i \le k, x\in \chi_i(A_i \setminus L_i \times (r_0t_i, \infty)) \\
\Th(s),  & 1 \le i \le k, x = \Psi^i_t(s, p) \text{ for some } p \in L_i \text{ and } s\in (r_0t_i,  R_0) \\
\Th(s),  & k < i \le l, x = \Psi^i_\CS(s, p) \text{ for some } p \in L_i \text{ and } s\in (0,  R_0) \\
\end{array}
\right.
\e
Here we mean by $\Th(f)$ something that is bounded on both sides by $f$, up to a constant	that is independent of the choice of $\bar{A}$. Furthermore over $\Psi^i_{\bar{A}}((r_0t_i,R_0)\times L_i)$ we require $\rho$ to be an increasing function of the radial component $s\in (r_0t_i, R_0)$.
\end{dfn}
Choose $\rho$ to be the distance in $M$ to the closest singular point of $N$ and modify away from the singular points such that the functions are bounded by $1$. This will be an example of a family of radius functions. From this we also see that we can choose the family to be smooth and have uniformly bounded derivative. We can now define alternative Sobolev-norms on $L^p_k$-spaces on $N^{\bar{A}}$ that take into account the scale of the glued pieces. Suppose $E$ is a metric vector bundle over $N^{\bar{A}}$ with a connection $\nabla$. Let $\bar{\de} \in \R^l$ be an arbitrary weights. We then define the $L^p_{k, \bar{\de},\bar{A}}$-norm of a section $s \in C^\infty(E)$ as: 
\e
 \nm{s}_{L^p_{k, \bar{\de}, \bar{A}}} = \left(\sum_{i = 0}^k\int_{N^{\bar{A}}}\md{\rho^{-w+i}\nabla^i s}^p \rho^{-4}\dvol\right)^{\frac{1}{p}} .
\e
Here $w: N^{\bar{A}} \longra \R$ is a smooth weight function that interpolates between the chosen weights near each singularity. If all singularities are removed, so that $N^{\bar{A}}$ is non-singular and compact, these norms are all uniformly equivalent for different values of $\bar{\delta}$,  but they are not uniformly equivalent in $\bar{A}$, in the sense that the comparison constant will be unbounded. As we reduce the global scale, these norms reduce over the glued pieces to the norms for conical manifolds we introduced above. This will allows us to transplant results for the conical parts $A_i$ and $N$ onto the glued $N^{\bar{A}}$. Near the singularities that we did not remove, this norm is exactly the weighted Sobolev norm for conically singular manifolds. We can define Hölder spaces that vary with ${\bar{A}}$, the $C^{k,\al}_{\de, \bar{A}}$-spaces, in a similar manner. We note that the Sobolev constants for different values of ${\bar{A}}$ will all be uniformly comparable.

\subsection{Estimates}
\label{4_3_Estimates}

Consider the approximate Cayleys $N^{\bar{A}}$ that we have defined above together with a family of radius functions $\rho$. For sufficiently small global scale $t$ we have by Proposition \ref{4.1_alpha_cayley} that $N^{\bar{A}}$ are $\alpha$-Cayley for any fixed $\alpha < 1$. Thus in particular $N^{\bar{A}}$ admits a canonical deformation operator as in \eqref{2_1_deformation_op}, which as in  \ref{2_3_cs_structure} can be augmented to include $\CS$ deformations of the unglued conical singularities as well as deformations of the $\Spin(7)$-structure: 
\e
F_{\bar{A}}: C^\infty_{\bar{\de}}(\nu_\ep(N^{\bar{A}})) \times \cO \longra C_{\bar{\de}-1}^\infty(E_{\cay}).
\e 
Here $\cO \subset \cS \times \cF$ is a small neighbourhood of the point which corresponds to the initial $\Spin(7)$-structure and the initial vertices and cones of $N^{\bar{A}}$. Moreover we define $\nu_\ep(N^{\bar{A}}) = \{(v,p) \in \nu(N^{\bar{A}}), \nm{v} < \ep \rho(p)\}$, similar to the $\CS$ and $\AC$ cases. The weights $\bar{\de} \in \R^l$ are chosen such that for $1 \le i \le k$, $1 < \de_i < \frac{\mu_i (\la-2)+1}{\la-\mu_i}$ and for $k+1 \le i \le l$ we set $\de_i = \mu_i$. We will explain later how this condition arises. Let us just note for now that $\lim_{\la \ra -\infty}\frac{\mu_i (\la-2)+1}{\la-\mu_i} = \mu_i$, thus the formula is compatible with "gluing a Cayley cone" (which is $\AC_\la$ for any $\la < 0$) onto a conical singularity, which should do nothing to the singularity. Also, increasing the rate of the asymptotically conical piece makes the allowed ranges of rates for the gluing smaller. In the following we will write $\mu = \min_i(\mu_i)$ and $\de = \min_i(\de_i)$. This may seem like a restriction, however as by assumption $(1, \mu_i] \cap \cD_{L_i} = \emptyset$, we do not lose anything by doing this. Any $\CS_{\mu}$ manifold can be improved to be $\CS_{\mu_i}$ as no critical weights are present in the range $(\mu, \mu_i)$ (\cite[Thm. 5.5]{joyceReg}) We denote the linearisation of $F_{\bar{A}}$ at $0$ by $D_{\bar{A}}$. We can now establish bounds on the glued deformations operators, using our results for the $\CS$ and $\AC$ case. In particular we will take into account the dependence of various constants on the parameter ${\bar{A}}$. This will be important later when we deform all the $N^{\bar{A}}$ simultaneously to become Cayleys. In this regard, the most important property of the deformation operator is its dependence on $N$, $v$ and $\Phi$. In particular we have pointwise dependence only on $p,v(p), \nabla v(p)$ and $T_pN$ as in the following proposition:
\begin{prop}[{\cite[Lem. 4.2]{englebertDeformations}}]
\label{3_2_F_shape}
The deformation operator on $N^{\bar{A}}$ for the varying $\Spin(7)$-structures and cone configuration can be written as follows, for $v \in C^\infty_{loc}(\nu(N^{\bar{A}}))$, $s \in \cO$ and $p\in N^{\bar{A}}$:
\ea
F_{\bar{A}}(v, s)(p) &= \boF(p, v(p), \nabla v(p), T_pN^{\bar{A}}, s)\nonumber \\
&=  F_{\bar{A}}(0, s)(p)+ D_{{\bar{A}}, s}v(p) + \boQ(p, v(p), \nabla v(p), T_pN^{\bar{A}}, s).
\ea
Here $D_{{\bar{A}}, s}$ is the linearisation of $F_{\bar{A}}(\cdot, s)$ at $0$ and $\boF,\boQ$ are smooth fibre-preserving maps: 
\eas
\boF, \boQ: TM_\eps \times_N (T^*M \otimes TM )_\eps \times \Cay_{\al}(M) \times \cO  \longra \boE_{\cay},
\eas 
where $\boE_{\cay} = \{(p, \pi, e): (p, \pi) \in \Cay_{\al}(M), e \in E_\pi\}$ and $\al$ is sufficiently large. Here we see both sides as fibre bundles over $\Cay_{\al}(M) \times \cO$. We define the map $Q_{\bar{A}}: C^\infty(\nu_\ep(N^{\bar{A}})) \times \cO \longra C^\infty(E_{\cay})$ as $Q_{\bar{A}}(v, s) = F_{\bar{A}}(v, s)-D_{{\bar{A}}, s}v$.
\end{prop}
We stress that the smooth maps $\boF$ and $\boQ$ only depend on the family of $\Spin(7)$ structures $\Phi_s$, and not on the Cayley submanifold. The term $Q_{\bar{A}}$ contains the contributions of $v$ and $\nabla v$ which are quadratic and higher. Since $N^{\bar{A}}$ is both conically singular and has nonsingular regions of high curvature as the global scale decreases, we need to apply both the compact and the conically singular theory to prove the following:
\begin{prop}
\label{3_1_F_A_smooth}
Let $p > 4$, $k \ge 1$ and $1 < \de< \frac{\mu (\la-2)+1}{\la-\mu}$. Then the deformation map $F_{\bar{A}}$ is well-defined and $C^\infty$ as a map between Banach manifolds:
\e
F_{\bar{A}} : \cM_{\bar{A}} = \{ v\in L^p_{k+1, \de, {\bar{A}}}(\nu_\ep(N^{\bar{A}})) : \nm{v}_{L^p_{k+1,\de}} < \ep \} \times \cS \longrightarrow L^p_{k, \de-1}(E_{\cay}),
\e
whenever $\ep > 0$ is sufficiently small, and can be chosen the same for all  $\bar{A}$. Any $v \in L^p_{k+1, \de}(\nu_\ep(N^{\bar{A}}))$ such that $F_{\bar{A}}(v) = 0$ is smooth. Finally it is also Fredholm.
\end{prop}
The proof of the smoothness of $F_{\bar{A}}$ is essentially the same as for \ref{2_3_cs_structure}, with all the norms replaced by their appropriate counterparts. As in the  usual deformation theory it relies on separate estimates of the first few terms of the Taylor expansion of $F_{\bar{A}}$, where we will now need to take into account the dependence on ${\bar{A}}$. Next, as the Hölder space $C^{k, \al}_{\de, {\bar{A}}}$ for a fixed ${\bar{A}}$ can be seen as $C^{k, \al}_{\de}$ for a conically singular manifold, usual elliptic regularity arguments apply and and show smoothness, such as for instance in the proof of \ref{2_3_cs_structure}. Let us now in turn take a look at the constant, linear and quadratic estimates of $F_{\bar{A}}$ and pay close attention to the constants involved.

\subsubsection*{\texorpdfstring{Estimates for $\tau$}{Estimates for tau}}

We first investigate how well $N^{\bar{A}}$ approximates a Cayley as a function of the global scale $t$. Our main result will be that a priori $N^{\bar{A}}$ should converge to an ideal Cayley in $C^\infty_{\de, {\bar{A}}}$ for $1 < \de < \frac{\mu (\la-2)+1}{\la-\mu}$, uniformly in $\bar{A}$.

\begin{prop}[Pointwise estimates]
\label{4.2_pointwise_tau}
Denote by $g^{\bar{A}}$ the Riemannian metric on $N^{\bar{A}}$ coming from the embedding into $M$. For $t$ sufficiently small and for $s \in \cS$ sufficiently close to our initial $\Spin(7)$-structure, we have the following estimates on the derivative $\nabla^k\tau|_{N^{\bar{A}}}$ for $k \ge 0$:
\ea
\md{\chi_i^*\tau|_{N^{\bar{A}}}}(x)&\lesssim \md{x},  \label{4.2_tau_N_l_0}\\
\md{\nabla^k\chi_i^*\tau|_{N^{A_i}}}(x)&\lesssim t_{i}^{-k+1}, \quad k \ge 1 \label{4.2_tau_N_l_12}\\
\md{\nabla^k\tau|_{N^{A_i}_m}} &\lesssim  t_i^{-\la}\rho^{\la-k-1} + \rho^{\mu-j-1} ,\quad \rho \in (r_0t_i, \frac{1}{4}t_i^\nu) \label{4.2_tau_N_m_u} \\
\md{\nabla^k\tau|_{N^{A_i}_m}} &\lesssim  t_i^{-k\nu}(t_i^{(\nu-1)(\la-1)}+ t_i^{\nu(\mu-1)}) ,\label{4.2_tau_N_m_l} \\
\md{\nabla^k\tau|_{N^{\bar{A}}_u}} & \lesssim d(s, s_0). \label{4.2_tau_N_u} 
\ea
Here $\nabla$ and $\md{\cdot}$ are computed with respect to $\chi^*g^{\bar{A}}$ in the first line, and $g^{\bar{A}}$ in the last two lines. Furthermore, the constants hidden in the $\lesssim$-notation are independent of $\bar{A}$. 
\end{prop} 
\begin{proof}
We adapt the method of proof from Proposition 8.1. in \cite{lotayDesingularizationCoassociative4folds2008}. Note first that $N^{\bar{A}}_u$  are Cayley by construction for our initial $\Spin(7)$-structure, and therefore $\tau_{s_0}$ and all its derivatives vanish on them. As $N^{\bar{A}}_u$ is compact, we can easily get the bound \eqref{4.2_tau_N_u}. 

Consider next $N^{\bar{A}}_l$. In what follows we can think of the conically singular points of $N^{\bar{A}}$ as being obtained by gluing in a Cayley cone, and thus can be treated no differently from the desingularised regions. We have by Taylor's theorem that:
\eas
\md{\chi_i^*\tau}(x) = \md{\chi_i^*\tau}(0) + O(\md{x}).
\eas
We have chosen $\chi_i$ to be a $\Spin(7)$-coordinate system, so that $\chi_i^*\Phi(0) = \Phi_0$, where $\Phi_0$ is the standard Cayley form on $\R^8$. We therefore also have $\chi_i^*\tau(0) = \tau_0$, where $\tau_0$ is the standard quadruple product on $\R^8$. Now since $\chi_i^{-1}(N^{\bar{A}}_l)$ is Cayley with respect to $\Phi_0$, we get that: 
\eas
\md{\chi_i^*\tau|_{N^{\bar{A}}_l}}(x) = \md{\tau_0|_{N^{\bar{A}}_l}}(0) + O(\md{x}) = O(\md{x}).
\eas
Thus we get \eqref{4.2_tau_N_l_0}. Now for $k\ge 1$, we have $\md{\nabla^k \tau_0|_{{N^{\bar{A}}_l}}} = 0$, as the $A_i$ are Cayley for $\Phi_0$. So we would like to bound: 
\eas
\md{\nabla^k (\chi_i^*\tau-\tau_0)|_{{N^{\bar{A}}_l}}}.
\eas 

For $t > 0$, think of $tA_i$ as a map $f_t: A_i \longra \R^8 \times \La^4$ which maps $p \in A_i \longmapsto (tp, T_pA_i)$, and of $\chi_i^*\tau-\tau_0$ as a map $\tilde{\tau}: \R^8 \times \La^4 \longmapsto \La^2_7$ with the property that $\tilde{\tau}(0, \omega) = 0$.  We therefore have a Taylor expansion for small $v \in \R^8$:
\eas
\tilde{\tau}(v, \omega) = L_\om[v]+ R_{\om, v}[ v \otimes v]. 
\eas
Here $L_\om$ and $R_{\om, v}$ are linear maps that depend smoothly on $\omega \in \La^4$, and $R_{\om, v}$ also depends smoothly on $v$. Thus, we see that: 
\eas
\tilde{\tau} \circ f_t(p) = t L_{T_pA}[ p] + t^2R_{T_pA, p} [ p \otimes p],
\eas 
From which we can deduce that:
\eas
\nabla_\xi (\tilde{\tau} \circ f_t)(p) &= t(L_{T_pA} [ \xi] + \D L_{T_pA} [ p, \nabla_\xi T_pA ]) \\
&+  t^2(2 R_{T_pA, p} [ \xi \otimes p] + \D R_{T_pA, p} [ p\otimes p, \nabla_\xi T_pA, \xi ]). 
\eas
The linear maps and their derivatives can be bound uniformly, as both $p \in B_\eta(0)$ and $T_pA$ vary in compact sets. Thus we see that:
\eas
\md{\nabla (\tilde{\tau} \circ f_t)(p)} \le C(A, \tau)(t + t^2(\md{p} + \md{p}^2 \md{\nabla T_pA})) \le C(A, \tau)t.
\eas
Here we used the fact that $\md{\nabla T_pA} \in O(\md{p}^{-1})$ and $\md{p}\in O(1)$. Thus going back to our original situation, after rescaling by $t_i$ to account for the fact that the metric on $tA_i$ scales as well, we obtain: 
\eas
\md{\nabla^k \chi_i^*\tau|_{{N^{\bar{A}}_l}}} = \md{\nabla^k (\chi_i^*\tau-\tau_0)|_{{N^{\bar{A}}_l}}} \lesssim 1.
\eas
The higher derivatives can be deduced the same. The key point is that naively rescaling will lead to a factor $t_i^{-k}$, but since the $A_i$ are Cayley, we can improve it by one factor of $t_i$ via the above Taylor expansion argument. 

Finally, we consider $N^{\bar{A}}_m$, where the interpolation happens and where we also expect the biggest error to appear. We will consider $(\Psi^i_{\bar{A}})^* \tau |_{N^{\bar{A}}_m}$, which is a form on the cone portion $C = (r_0 t_i, t_i^\nu) \times L$, and we will prove the analogue of \eqref{4.2_tau_N_m_u} and \eqref{4.2_tau_N_m_l} with respect to the cone metric. Now as $t\ra 0$, the pullback metric $(\Psi^i_{\bar{A}})^*g^{\bar{A}}$ will converge uniformly in $t$ to the conical metric. In particular, the conical metric and the pullback metrics for small $t$ are all uniformly equivalent with proportionality factors independent of the global scale. Thus all quantities of the form $\md{\nabla^k s}$, computed with regards to any of these metrics, will be in the same asymptotic class. Denote by $\iota: C \ra \R^8$ the embedding of the cone. We then have that: 
\eas
\md{\nabla^k (\Psi^i_{\bar{A}})^*\tau|_{N^{\bar{A}}_m}} &= \md{\nabla^k (\Th^i_{\bar{A}})^*\chi_i^*\tau|_{N^{\bar{A}}_m}} \\
&\le  \md{\nabla^k(\Th^i_{\bar{A}} -\iota)^*\chi_i^*\tau}+\md{\nabla^k\iota^*\chi_i^*\tau}
\eas
Upper bounds for the second term can be given in an analogous way to what we have done for $N^{\bar{A}}_l$, as the cone is Cayley. As the cone is scaling invariant, and we are interested in the region with radius in $(r_0t_i, t_i^{\nu})$, we can run the same argument while only rescaling by $t_i^{\nu}$, and thus we get an error of $t_i^{-k\nu}$. We will see below that this is always the asymptotically better term. Now onto the remaining term. Notice that $\chi_i^*\tau$ is a fixed quantity, and the only dependence on $\bar{A}$ is within $\Th^i_{\bar{A}}-\iota$. So let us more generally bound: 
\eas
\md{\nabla^k f^*\omega},
\eas
for $f: C \ra \R^8$ a smooth function, and some form $\omega\in \Omega^k$. From the definition of pullback we see that there are smooth maps $E_k$, independent of $f$ such that:
\e
\label{4.2_rescaling_f}
\nabla^k f^*\omega(p) = E_k(f(p), \nabla f(p), \dots, \nabla^{k+1} f(p)).
\e
These maps have the additional property that they are affine in $\nabla^k f(p)$ where $k \ge 1$. Consider the scaling behaviour of both sides when the cone is rescaled by $\ga > 0$. In other words, we replace $f$ by $f_\ga$, such that $f_\ga(p) = f(\ga\cdot p)$. Equation \eqref{4.2_rescaling_f} still holds for $f_\ga$, and we can relate the norms of both sides to the corresponding terms for $f$ as follows: 
\eas
\ga \cdot \nabla^k f^*(\chi^*\tau)(\ga\cdot p) &= \nabla^k f^*(\ga\cdot  \chi^*\tau)(p) \\
&= \nabla^k f_\ga^*(\chi^*\tau)(p)  \\
&= E_k(f_\ga(p), \nabla f_\ga(p), \dots, \nabla^{k+1} f_\ga(p)) \\
&= E_k(\ga\cdot f(p), \ga\cdot \nabla f(p), \dots, \ga\cdot \nabla^{k+1} f(p))
\eas
As the maps $E_k$ are affine in the (higher) covariant derivatives of $f$, we see that: 
\e
\label{4.2_taylor_f}
\ga^k\md{\nabla^k f^*(\chi^*\tau) } \lesssim \md{f(p)} +  \ga\md{ \nabla f(p)} + \dots +  \ga^{k+1}\md{\nabla^{k+1} f(p)}. 
\e
Let us now estimate the norms of $f$ and its derivatives. We have:
\ea
f(p, s) = (\Th^i_{\bar{A}}-\iota) (p,s) = &(1-\phi)(2t_i^{-\nu}-1)(\Th^i_\AC(p, s)-\iota(p, s))\nonumber \\
\label{4.2_th_iota}
+& \phi(2t_i^{-\nu}-1)(\Th_\CS(p, s)-\iota(p,s)).
\ea
We would like bounds on $f$ that are invariant when changing $\bar{A}$. Dealing with changes that keep the scales fixed is not a problem, as elements of a given scale form a compact space. Thus we are only worried about changes of scale. We hence apply the $\AC$-condition to the $t_i^{-1}A_i$, and rescale to obtain: 
\ea
\md{\nabla^k\Th_\AC(p, s)-\iota(p,s)}& = O(t_i^{-\la+1} s^{\la-k}), 
\ea
where the constant is independent of the scale. We get from the $\CS$ condition that: 
\ea
\md{\nabla^k\Th_\CS(p, s)-\iota(p,s)}& = O(s^{\mu -k}), 
\ea
Taken together we obtain the bound: 
\eas
\md{(\Th_t-\iota) (p,s)} \lesssim t_i^{-\la+1} s^{\la} + s^{\mu} . 
\eas
One can obtain bounds for the derivatives of $\Th_t-\iota$ in a similar manner. To be more explicit, the covariant derivatives applied $k$ times to \eqref{4.2_th_iota} will hit both $\phi(2t_i^{-\nu}-1)$ and $\Th_{\AC /\CS}-\iota$. If it hits $\phi$ a total of $l$ times in a term, we obtain a bound of the form $O(t_i^{-l\nu}\partial^l\phi \md{\nabla^{k-l}(\Th_{\AC/\CS}-\iota)})$. An explicit calculation leads us to the general formula: 
\eas
\md{\nabla^k(\Th_{\bar{A}}-\iota) (p,s)} = O\left(\sum_{j+l = k}  (t_i^{-\la+1}s^{\la-j} + s^{\mu-j})t_i^{-l\nu}\partial^l\phi \right) .
\eas 
Thus we can plug this into our estimate \eqref{4.2_taylor_f} to obtain the bound: 
\eas
\md{\nabla^k(\Th_{\bar{A}}-\iota)^*\chi^*\tau} = O\left(\sum_{j+l = k}  (t_i^{-\la+1}\rho^{\la-j-1} + \rho^{\mu-j-1})t_i^{-l\nu}\partial^l\phi\right).
\eas
From this we obtain the claimed bounds by noting that either $r_0 t_i \le \rho \le \frac{1}{4}t_i^\nu$, where $\partial \phi = 0$, or $\rho > \frac{1}{4}t_i^\nu$, so that $ \rho = O(t_i^{\nu})$. 
\end{proof}

\begin{prop}[Initial Error estimate]
For a sufficiently small global scale $t>0$ and for $s \in \cS$ sufficiently close to our initial $\Spin(7)$-structure, $p > 4$, $\de > 1$, $\nu = \frac{\la-1}{\la-\mu}$, $k \in \N$, we have:
\ea
\label{4.2_tau_estimate}
\nm{F_{\bar{A}}(0,s)}_{L^p_{k, \de-1, \bar{A}}} &< C_F (t^{-\de\nu} (t^{\nu\mu} + t^{(\nu-1)\la+1}) + d(s, s_0)) \\
&< C_F (t^{\nu(\mu-\de)} + d(s, s_0))  \nonumber
\ea
Here $C_F > 0$ is a constant that only depends on the geometry of $(M, \Phi_s)$ and not on $\bar{A}$.
\end{prop}

\begin{proof}
Let $0 \le j \le k$. Subdivide $N^{A_i}_{m} = N^{A_i}_{m, 1} \cup N^{A_i}_{m, 2}$, where $N^{A_i}_{m, 1} $ is the region where $\rho \le \frac{1}{4}t_i^\nu$ and $N^{A_i}_{m, 2}$ the rest.  We then have that: 
\eas
&\int_{N^{\bar{A}}} \md{\rho^{-\de+1+j}\nabla^j F_{\bar{A}}}^p \rho^{-4} \dvol = \int_{N^{\bar{A}}} \md{\rho^{-\de+1+j}\nabla^j \tau|_{N^{\bar{A}}}}^p \rho^{-4} \dvol \\
&\lesssim \sum_{i=0}^l\int_{N^{A_i}_l} (\rho^{-\de+1 + j} \md{\chi_i^*\tau|_{N^{A_i}}})^p\rho^{-4} \dvol  
+\sum_{i=0}^l\int_{N^{A_i}_m} (\rho^{-\de+1 +j} \md{\nabla^k\tau|_{N^{A_i}_m}})^p \rho^{-4} \dvol\\
&+ \vol(N^{\bar{A}}_u)d^p(s, s_0) \\
&\lesssim \sum_{i=0}^l\int_{N^{A_i}_l} (\rho^{-\de+1  +j} t_{i}^{-j+1})^p \rho^{-4}\dvol +\sum_{i=0}^l\int_{N^{A_i}_{m, 1}} (\rho^{-\de +1+j} (\rho^{\mu-j-1}+t_i^{-\la+1}\rho^{\la-j-1}))^p \rho^{-4}\dvol \\
&+ \sum_{i=0}^l\int_{N^{A_i}_{m, 2}} (t_i^{-\de\nu +j\nu+\nu } t_i^{-j\nu}(t_i^{\nu(\mu-1)}+t_i^{(\nu-1)(\la-1)}))^p \rho^{-4}\dvol +d^p(s, s_0) \\
&\lesssim \sum_{i=0}^l t_{i}^{p(2-\de)}\int_{N^{A_i}_l} \rho^{-4} \dvol
+\sum_{i=0}^l \int_{N^{A_i}_{m, 1}} \rho^{-p\de} (\rho^{\mu}+\rho^{(1-\frac{1}{\nu})\la+\frac{1}{\nu} })^p  \rho^{-4} \dvol \\
&+\sum_{i=0}^l t_i^{-p\nu\de}(t_i^{\nu\mu}+t_i^{(\nu-1)\la +1})^p  \int_{N^{A_i}_{m, 2}}  \rho^{-4} \dvol+d^p(s, s_0) 
\eas
\eas
&\lesssim \sum_{i=0}^l (t_i^{p(2-\de)} + t_i^{-p\nu\de}(t_i^{\nu\mu}+t_i^{(\nu-1)\la+1})^p) 
+\sum_{i=0}^l \int_{N^{A_i}_{m, 1}} \rho^{p(\mu-\de)} \rho^{-4} \dvol+d^p(s, s_0) \\
&\lesssim \sum_{i=0}^l( t_i^{-p\nu\de}(t_i^{\nu\mu}+t_i^{(\nu-1)\la+1})^p + t_i^{-p\nu\de}t_i^{p\nu\mu}) +d^p(s, s_0)\\
 &\lesssim \sum_{i=0}^l t_i^{-p\nu\de}(t_i^{\nu\mu}+t_i^{(\nu-1)\la+1})^p +d^p(s, s_0).
\eas
Here we used all the various bounds from Proposition \ref{4.2_pointwise_tau} as well as the fact that we chose. Moreover we used that and that $\rho$ can be uniformly bound from above by $2t_i^\nu$ on both $N^{A_i}_l$ and $N^{A_i}_m$. Furthermore we can also bound $\rho$ from below by $\frac{1}{4}t_i^\nu$ on $N^{A_i}_{m, 2}$, and by $r_0t_i$ on $N^{A_i}_{m, 1}$ and $N^{A_i}_{l}$. The integral:
\eas
\int_{N^{A_i}_{l}}  \rho^{-4}\dvol \lesssim \int_{\eps}^{r_0} s^{-4}s^3 \d s  \le C
\eas
is bounded independently of $t_i$, as is:
\eas
\int_{N^{A_i}_{m,2}}  \rho^{-4}\dvol \lesssim \int_{\frac{1}{4}t_i^{\nu}}^{2t_i^{\nu}} s^{-1} \d s =\log( 2t_i^{\nu}) -\log\left( \frac{1}{4}t_i^{\nu}\right) \le C.
\eas
Finally, we compute that
\eas
\int_{N^{A_i}_{m, 1}} \rho^{p(\mu-\de)} \rho^{-4}\dvol \lesssim \int_{r_0t_i}^{\frac{1}{4}t_i^{\nu}} s^{p(\mu-\de)-1} \d s  \lesssim t_i^{p\nu(\mu-\de)} - t_i^{p(\mu-\de)}\lesssim t_i^{p\nu(\mu-\de)} .
\eas
The bound now follows as the exponent of the $t_i$ is positive, and thus the biggest one dominates, which is the global scale $t$. For the second line in \ref{4.2_tau_estimate} we use our choice of $\nu =  \frac{\la-1}{\la-\mu}$, which is chosen exactly so that $\nu\mu = (\nu-1)\la+1$.
\end{proof}

\subsubsection*{\texorpdfstring{Estimates for $D_{\bar{A}}$}{Estimates for D_{\bar{A}}}}

Recall that in our construction of $N^{\bar{A}}$, we have assumed identical cones (as subsets of $\R^8$) for the pieces $A_i$ and $N$, given the choice of a $\Spin(7)$-coordinate system. Here the interpolation happened between the radii $\ha t^\nu$ and $t^\nu$, where $0 < \nu < 1$ is a constant. In order to derive estimates similar to the bounds in Propositions \ref{2_3_cs_bound_k} and \ref{2_3_ac_bound_k} for $D_{\bar{A}}$, we use a partition of unity to combine the results for the parts. For this we need further constants $0 < \nu'' < \nu' < \nu < 1$. Let $\tilde{\phi}: \R \ra [0,1]$ be a smooth cut-off function, such that:
\e
\tilde{\phi}|_{(-\infty, \nu'']} = 0, \quad \tilde{\phi}|_{[\nu', +\infty) } = 1.\nonumber
\e
Using $\tilde{\phi}$ we define a partition of unity on $N^{\bar{A}}$ as follows. Let $t > 0$ be the global scale of $N^{\bar{A}}$ and suppose that for $t_i$ is the local scale of $A_i$. We then define:

\e
\label{4.2_alpha_t}
\alpha(p) = \left\lbrace \begin{array}{cc}
\tilde{\phi}\left(\frac{\log(\rho(p))}{\log(t_i))}\right),&\text{ if } p \in \Psi^{\bar{A}}_i(L_i \times (r_0t_i, R_0)) \\ 
0,& \text{ if } p \in N^{\bar{A}}_u, \\
1,& \text{ if } p \in N^{\bar{A}}_l, \\
\end{array} \right.
\e
Then $\al(p) = 0$  on $\Psi^{\bar{A}}_i(L_i \times (r_0t_i, R_0))$ if $\rho(p) \ge t_i^{\nu''}$ and $\al(p) = 1$ if $\rho(p) \le t_i^{\nu'}$. Thus $\al|_{N\Psi^{\bar{A}}_i(L_i \times (r_0t_i, R_0))}$ is supported in $N^{A_i}_{\AC} = \Psi^i_{\bar{A}}(L_i \times (r_0t_i, t_i^{\nu''}))$ and $1-\al$ is supported in $N^{\bar{A}}_{\CS} =  \Psi^{\bar{A}}_i(L_i \times (t_i^{\nu'}, R_0))$. In particular the gluing region $N^{\bar{A}}_m$ is entirely contained in $N^{\bar{A}}_{\AC}$. We would now like to relate the operator $D_{\bar{A}}|_{N^{A_i}_{\AC}}$ to  $D_\AC$ on a perturbation of $A_i$ and the  operator $D_{\bar{A}}|_{N^{\bar{A}}_{\CS}}$ to $D_\CS$ on $N$. 
\begin{figure}
\begin{center}
\includegraphics[scale=0.7]{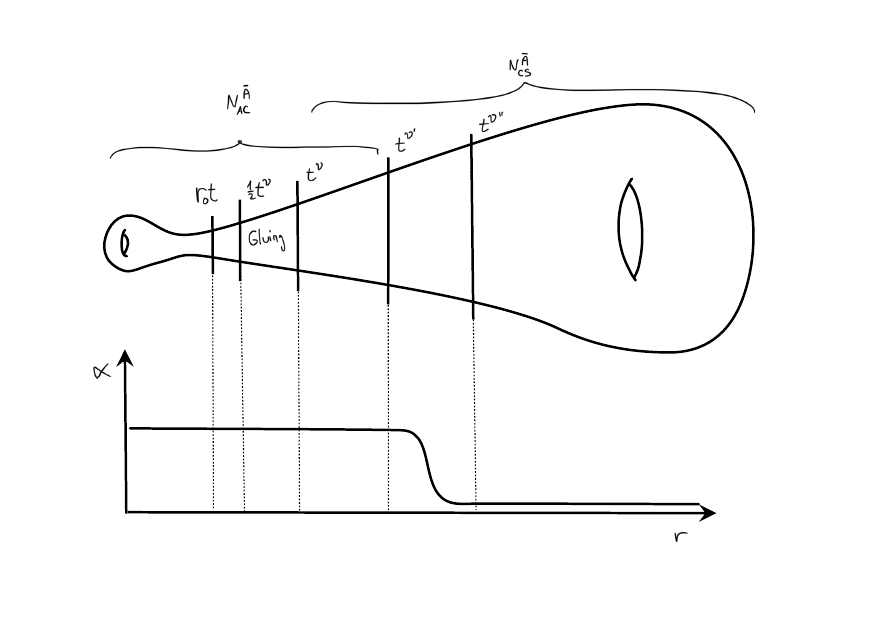}
\end{center}
\caption{Decomposition of the glued cone}
\label{3_4_alpha_behaviour}
\end{figure}
To do this we define a pseudo-kernel $\kappa_{\bar{A}} \subset C^\infty(\nu_\ep(N^{\bar{A}}))$ for the glued operator, the analogue of $\kappa_{\CS}$ and $\kappa_{\AC}$ from Propositions \ref{2_3_cs_pseudo} and \ref{2_3_ac_pseudo} respectively. We will be working with a rate $1 < \de < \frac{\mu (\la-2)+1}{\la-\mu}$ such that $(1, \de] \cap \cD_{L_i} = \emptyset$ for all the links of $N$. The space   $\kappa_{\bar{A}}$ will be defined as a direct sum of contributions from both pieces. First, the treatment of the conically singular piece is immediate. The elements of $\ka_{\CS}$ all have support in a compact subset of $N$. Thus for $t$ sufficiently small, we can consider them as section of $N^{\bar{A}}_{\CS}$ directly, since the $N^{\bar{A}}_{\CS}$ exhaust $N$ as the global scale goes to $0$. In particular, we can then also consider them as sections over $N^{\bar{A}}$ by extending by $0$ over $N^{\bar{A}}\setminus N^{\bar{A}}_{\CS}$. Even more, for small enough $t$, the elements of $\kappa_{\CS}$, seen as sections on $N^{\bar{A}}$, vanish on $N^{\bar{A}}_{\AC}$. Similarly the operators $D_{\CS}$ and $D_{\bar{A}}|_{N^{\bar{A}}_{\CS}}$ can be identified and the bounds for $D_{\CS}$ carry over.

Next, the interpretation of $\ka_\AC$ is more delicate, as the gluing region is a perturbation of $A_i$. We first find an identification between $N^{A_i}_{\AC}$ and an open subset of $A_i$ as follows. For technical purposes we fix a further rate $0 < \tilde{\nu} < \nu''$. Then there is a diffeomorphism  between an open subset $\tilde{A}_i \subset A_i$ and $\chi(K_{A_i}) \sqcup \Th^{i}_{\bar{A}}(L_i\times (r_0t_i,t_i^{\tilde{\nu}})) $, given by sending: 
\eas
p \in K_A &\longmapsto \chi(p), \\
\Psi_\AC\left(p,s\right) &\longmapsto\Th^{i}_{\bar{A}}(p,s) .
\eas
Let us call this map $\Psi^{i}_{\bar{A}} : \tilde{A}_i \ra N^{\bar{A}}_{\AC}$, which as usual factors as $\Psi^{i}_{\bar{A}} = \chi \circ \Th^{i}_{\bar{A}}$. As the operator $D_{\AC}$ not only takes into account the metric structure of $A$, but also the ambient $\Spin(7)$-structure we now thicken the map $\Th^{i}_{\bar{A}}$. Let $U^{i}_{\bar{A}}$ be a tubular neighbourhood of $\tilde{A}_i$ in $\R^8$, so that every $q \in U^{i}_{\bar{A}}$ can be written uniquely as $q = p+v$, where $p \in \tilde{A}_i$ and $v \in (\nu_{\eps, \Phi_0}(\tilde{A}_i))_p$. We now define:
\eas
\tilde{\Th}^{i}_{\bar{A}} : U^{i}_{\bar{A}} &\longra \R^8 \\
(p,v)  &\longmapsto \Th^{i}_{\bar{A}}(p)+v 
\eas
Then clearly $\tilde{\Th}^{i}_{\bar{A}}|_{\tilde{A}_i} = \Th^{i}_{\bar{A}}$. We will now transport the Cayley form in a vicinity of $N^{\bar{A}}_{\AC}$ over to $A$. Consider first $\chi^* \Phi$, which is a four-form on $B_{r_0}(0) \subset \R^8$. Via pullback, we obtain a form, which we define pointwise as:
\eas
(\tilde{\Phi}^i_{\bar{A}})_p = t_i^{-4}\tilde{\Th}^{i}_{\bar{A}} \chi^* \Phi(t_ip)\in \Omega^4(U^t_A).
\eas
We introduced the factor $t_i^{-4}$ to counteract the rescaling by $t_i$. With this normalisation we have $\tilde{\Phi}^i_{\bar{A}} \ra \Phi_0$ uniformly on $U^t_A$ as $t \ra 0$. This follows from Taylor's theorem and the fact that $U^t_A \subset B_{2r_0t^{\tilde{\nu}-1}}(0)$, since it gives $\chi^*\Phi(t_ip)-\Phi_0 = O(t_i\md{p})$ as we have $\chi^*\Phi(0) = \Phi_0$. We now extend $\tilde{\Phi}^i_{\bar{A}}$ to a form $\Phi^i_{\bar{A}}$ defined on all of $\R^8$. For this recall the smooth cut off function $\phi: \R \ra [0,1]$ which we used in the construction of $N^{\bar{A}}$. It vanishes for negative values and is equal to $1$ for values $\ge 1$, as in \eqref{4.1_cutoff}. The space of Cayley forms on $\R^8$ is a smooth submanifold $\cC \subset \La^4 \R^8$ of dimension $43$. Choose local coordinates $c: B_1(0) \subset \R^{43} \ra \cC$ such that $c(0) = \Phi_0$. As we have uniform convergence $\tilde{\Phi}_t \ra \Phi_0$ on $U^t_A$, we will eventually have $\tilde{\Phi}_t \in \im c$ for $t$ sufficiently small. We then interpolate between $\tilde{\Phi}_t$ and $\Phi_0$ between the radii $\ha t^{\tilde{\nu}-1}$ and $t^{\tilde{\nu}-1}$ as follows: 
\e
\label{4_2_structurOnAC}
\Phi^i_{\bar{A}}(\Th_{\AC,t_i^{-1}A_i}(r, p)+v) = c(c^{-1}(\tilde{\Phi}^i_{\bar{A}}(\Th^i_\AC(r, p)+v)) \phi(2rt_i^{-1+\tilde{\nu}}-1)).
\e

We now have a family of forms $\Phi^i_{\bar{A}}$ on $U_A$. If we choose the global scale sufficiently small, we can extend these forms to all of $\R^8$. For sufficiently small $t$, we have that $A$ is almost Cayley. These forms $\Phi^i_{\bar{A}}$ form a continuous family with respect to the parameter $\bar{A}$, and as $t\longra 0$, we get uniform convergence $\Phi^i_{\bar{A}} \longra \Phi_0$. In fact, we even have $C^\infty_\eta$-convergence.

\begin{lem}
The family $(\R^8, \Phi^i_{\bar{A}})_{\bar{A} \in\overline{\cM}^\la_{\AC}(A_i)}$ is a continuous family of $C^\infty_\eta$ perturbations of the standard $\Spin(7)$ form $\Phi_0$. The rate $\eta < 1$ only depends on  the constant $0 < \tilde{\nu} <1$ chosen for the gluing, and $\eta \ra -\infty$ as $\tilde{\nu} \ra 1$ . For $\eta < \la$ , we have that $A$ is an $\AC_\la$-submanifold for the $\Spin(7)$-structure $\Phi^i_{\bar{A}}$.
\end{lem}
\begin{proof}
Note that the family $\Phi^i_{\bar{A}}$ is flat at large radii, but the cut off radius $Ct_i^{\tilde{\nu}-1}$ depends on $\bar{A}$. Thus the deformations at non-zero global scale are compactly supported near a fixed $\bar{A}_0$, and in particular also in $C^\infty_\eta$ for any $\eta \le 1$. In particular, for any $\eta < \la$ the submanifold $A\subset \R^8$ will be $\AC_\la$ for $\Phi^i_{\bar{A}}$ because it is for $\Phi_0$. From Equation \eqref{4_2_structurOnAC} and $\md{\Phi_p - \Phi_{z_i}} = O(\rho) $ on $M $ we see that:
\eas
\md{(\Phi^i_{\bar{A}} - \Phi_0) r^{-\eta+1}} \le C t_i^{\tilde{\nu}}(t_i^{\tilde{\nu}-1})^{-\eta+1} = C t_i^{\eta(1-\tilde{\nu}) +2\tilde{\nu}-1}.
\eas
Thus we have $C^0_\eta$ convergence as $t \ra 0$ when $\eta > \frac{2\tilde{\nu}-1}{\tilde{\nu}-1}$. Similar reasoning for higher derivatives shows that:
\eas
\md{\nabla^k \Phi^i_{\bar{A}} r^{-\eta+k+1}} \le C t_i^{(-\eta+1)(\tilde{\nu}-1) +k}.
\eas
Thus $C^\infty_\eta$ convergence follows immediately whenever we have $C^0_\eta$ convergence.
\end{proof}

Let us return to the question of defining the analogue of $\ka_\AC$ for $N^{A_i}_{\AC}$. After composing $\tilde{\Th}_i^{\bar{A}}$ with $\chi$ we have an identification $\Psi_A^t$ of open neighbourhoods of $A$ and $N^{\bar{A}}_{\AC}$. We can hence pull back the elements of $\ka_\AC$ to sections of $TM|_{N^{\bar{A}}_{\AC}}$ for $t$ sufficiently small. Note that we cannot in general require that they are normal sections. To remedy this we will first project $\ka_\AC$ onto $\nu_{\Phi^i_{\bar{A}}}(A)$, i.e. the normal bundle of $A$ with respect to the Cayley form $\Phi^i_{\bar{A}}$. Note that $\nu_{\Phi^i_{\bar{A}}}(A)$ is exactly identified with $\nu(N^{\bar{A}}_{\AC})$ under $\Psi_A^t$. Thus we define the space of sections $\ka_{\AC,\bar{A}}$ as $\ka_\AC$ projected onto $\nu_{\Phi^i_{\bar{A}}}(A)$ and then transported to $\nu(N^{\bar{A}}_{\AC})$. For $t$ sufficiently small the elements of $\ka_{\AC,\bar{A}}$ can be extended to sections on all of $N^{\bar{A}}$, and after further reducing $t$ the sections in $\ka_{\CS}$ and $\ka_{\AC,\bar{A}}$ will have disjoint support. In this case we define:
\e
\ka_{\bar{A}} = \ka_{\CS} \op \ka_{\AC,\bar{A}}.
\e
This is a family of pseudo-kernels for the family of operators $D_{\bar{A}}$. Note that $\ka_{\CS} $ contains all the contributions which have rate $ \ge \de$, and $\ka_{\AC,\bar{A}}$ all the ones which nave rate $\le \de$. As $\de$ is by assumption not critical, this accounts for every possible deformation exactly once. Note also that while the non-linear deformation operator of an $\AC$ Cayley does not have geometrical meaning when the rate $\la > 1$, the linearised operator can be defined for any rate.

We now show the analogue of Propositions \ref{2_3_ac_bound_k} and \ref{2_3_cs_bound_k} for the glued manifold $N^{\bar{A}}$. For this we need to introduce an inner product that interpolates between $L^2_{\de-\eps}$ on the $\AC$ region and $L^2_{\de+\eps}$ on the $\CS$ region. So we define for $u,v \in C^\infty(\nu(N^{\bar{A}})$:
\e
\label{3_2_inner_product}
\left\langle u, v\right\rangle_{\de\pm\eps} = \int_{N^{\bar{A}}} \left\langle u, v\right\rangle \rho^{w-4} \dvol.
\e
Here $w(p) = \de -\eps$ whenever $\rho(p) \le \ha t_i^{\nu}$ and $w(p) = \de +\eps$ whenever $\rho(p) \ge t_i^{\nu}$. It is clear that the following is true:
\begin{prop}
\label{4.2_D_A_inv_estimate}
For $t$ sufficiently small there is a constant $C_{\AC}$, independent of $\bar{A}$ such that for $v \in L^p_{k+1, \de, \bar{A}}(\nu(N^{\bar{A}}))$ with $\supp(v) \subset N^{\bar{A}}_{\AC}$ which is $L^2_{\de\pm\eps}$-orthogonal to $\ka_{\AC,\bar{A}}$ we have: 
\e
\label{4_2_D_A_inv_est_formula}
\nm{v}_{L^p_{k+1, \de, \bar{A}}} \le C_{\AC} \nm{D_{\bar{A}} v}_{L^p_{k, \de-1, \bar{A}}}.
\e
\end{prop}

We now turn back to our task of combining the bounds on $D_A$ and $D_N$ to get bounds on the inverse of  $D_{\bar{A}}$ modulo the pseudo-kernel. Recall the cut off function $\alpha : N^{\bar{A}} \ra [0,1]$ we defined in \eqref{4.2_alpha_t}. It has the following decay properties: 

\begin{lem}
\label{4.2_asymptotics_rho}
Let $l \ge 1$ be given. Then:
\e
\nm{\nabla^l \al}_{C^0} \in O\left(\rho^{-l} \log(t_i)^{-1}\right).
\e
\end{lem}
\begin{proof}
As the cutoff function $\tilde{\phi}$ is smooth and only varies on a compact set, each of its derivatives up to a given order $l$ remain bounded on all of $\R$. Similarly, it is clear that all derivatives up to order $l$ of $\rho$ are bounded on $N^{\bar{A}}_{\CS}$, independent of the scale, since $\rho$ agrees with a radius function on the conically singular $N$ on this part. Finally, through an argument similar to the one in Proposition \ref{4.2_D_A_inv_estimate} the same holds true on $N^{A_i}_{\AC}$, except close to the radius $r_0t_i$, where the smoothing happens. We will see however that this is not an issue. In geodesic normal coordinates $\{x^i\}_{i=1, \dots, 4}$ around $p\in N^{\bar{A}}$ we have for $v \in T_pN^{\bar{A}}$: 
\ea
(\Lie_v \alpha)(p) &= \Lie_v \tilde{\phi}\left(\frac{\log\rho(p) }{\log t_i }\right) \nonumber\\ 
&= \frac{\d}{\d s}\bigg\rvert_{s = 0} \tilde{\phi}\left(\frac{\log\rho(\exp_p(sv)) }{\log t_i }\right) \nonumber\\
&= \frac{1}{\rho \log t_i} \cdot \tilde{\phi}'\cdot (\Lie_v \rho)(p).\nonumber 
\ea
From this we see that $\Lie_v \alpha$ is bounded by $\frac{C}{\rho \log t_i} $, where $C$ is independent of $p$ and $t_i$. This is because whenever the derivative of $\rho$ might become unbounded, the derivative of $\tilde{\phi}$ vanishes. Similarly we obtain for $v, w \in T_pN^{\bar{A}}$: 
\ea
\nabla^2 \alpha &= \nabla\left(\d x^i \ot \Lie_{\partial_i} \tilde{\phi}\left(\frac{\log\rho }{\log t_i }\right)\right)\nonumber \\
&= (\d x^i \ot \d x^j)\Lie_{\partial_j}\left( \frac{1}{\rho \log t_i} \cdot \tilde{\phi}'\cdot (\Lie_{\partial_i} \rho)\right) \nonumber\\
&= (\d x^i \ot \d x^j) \frac{\tilde{\phi}'' \rho\log (t_i)- \tilde{\phi}'\Lie_{\partial_j} \log t_i}{(\rho \log t_i)^2} \frac{\d ^2}{\d s \d r} \rho(\exp_p(s\partial_i + r \partial_j)) \nonumber \\
&= (\d x^i \ot \d x^j) \frac{1}{\rho^2 \log(t_i)} C(\tilde{\phi}'', \tilde{\phi}', \partial_i \partial_j\rho, \partial_i \rho). \nonumber 
\ea
This proves the statement for $l = 2$. The general statement follows in a similar way.
\end{proof}

\begin{lem}
\label{4.2_product_formula}
Let $B$ be a bundle of tensors over $N^{\bar{A}}$. Then there is a constant $C_0$ which is independent of $t$, such that for sufficiently small $t$ and a section $u \in L^p_{k, \de, \bar{A}}(B)$  the following holds: 
\e
\nm{\al u}_{L^p_{k, \de, \bar{A}}} \le C_0\nm{u}_{L^p_{k, \de, \bar{A}}}.
\e
\end{lem}
\begin{proof}
We have: 
\eas
2^{-k}&\nm{\al u}_{L^p_{k, \de, \bar{A}}}^p = 2^{-k}\sum_{i = 0}^k  \int_{N^{\bar{A}}} \md{\rho^{i-\de}\nabla^i(\al u)}^p \rho^{-4} \dvol \\
&\le \sum_{0 \le j \le i \le k}  \int_{N^{\bar{A}}} \rho^{p(i-\de)} \md{\nabla^j \al}^p\md{\nabla^{i-j} u}^p \rho^{-4} \dvol \\
&= \al^p \nm{u}^p_{L^p_{k, \de, \bar{A}}} + \sum_{0 \le j \le i \le k-1}  \int_{N^{\bar{A}}} \rho^{p(1+i-\de)} \md{\nabla^{j+1} \al}^p\md{\nabla^{i-j} u}^p \rho^{-4} \dvol \\
&\le \nm{u}^p_{L^p_{k, \de, \bar{A}}} + \sum_{0 \le j \le i \le k-1}  \int_{N^{\bar{A}}}  \md{\rho^{j+1}\nabla^{j+1} \al}^p\md{\rho^{i-j-\de}\nabla^{i-j} u}^p \rho^{-4} \dvol \\
&\le \nm{u}^p_{L^p_{k, \de, \bar{A}}} + C\left(\sum_{i=0}^{k-1}  \int_{N^{\bar{A}}}  \md{\rho^{i-\de}\nabla^{i} u}^p \rho^{-4} \dvol\right) \left( \sum_{j = 0}^{k-1} \nm{\nabla^{j+1} \al \cdot \rho^{j+1}}_{C^0}^p\right)\\
&\le \nm{u}^p_{L^p_{k, \de, \bar{A}}} + \frac{C}{\md{\log t}^p}\left(\sum_{i=0}^{k-1}  \int_{N^{\bar{A}}}  \md{\rho^{i-\de}\nabla^{i} u}^p \rho^{-4} \dvol\right) \\
&\le \left(1+\frac{C}{\md{\log t}^p}\right)\nm{u}^p_{L^p_{k, \de, \bar{A}}}.
\eas

Here we used the asymptotic behaviour of $\nabla^l \al$ from Proposition \eqref{4.2_asymptotics_rho} in the second to last line. 
\end{proof}

\begin{lem}
\label{4.2_bilinear_nabla_product_formula}
Let $B$ be a bundle of tensors over $N^{\bar{A}}$. Let $\diamond:  T^*N^{\bar{A}} \ot B \ra B$ be a family of bilinear pairings which have bounded norms as $\bar{A}$ varies, seen as sections of $T^*N^{\bar{A}} \ot B \ot B^*$. Then there is a constant $C_1 > 0$, independent of $\bar{A}$, such that for small enough $t$ and for any section $u \in L^p_{k, \de, \bar{A}}(B)$ we have: 
\e
\nm{\nabla \al \diamond u}_{L^p_{k, \de-1, \bar{A}}} \le \frac{C_1}{\md{\log t}}\nm{u}_{L^p_{k, \de, \bar{A}}}.
\e
\end{lem}
\begin{proof}
Using Proposition \ref{4.2_asymptotics_rho} the statement reduces to proving the following: 
\e
\nm{\nabla \alpha \diamond u}_{L^p_{k, \de-1, \bar{A}}} \le C\left(\sum_{i = 0}^{k-1}\nm{\nabla^{i+1} \al \rho^{i+1}}^p_{C^0}\right)^{1/p} \nm{u}_{L^p_{k, \de, \bar{A}}} \nonumber
\e
This in turn is proven similarly to the previous proposition.
\eas
\nm{\nabla& \alpha \diamond u}_{L^p_{k, \de+1, t}}^p  = \sum_{i=0}^k \int_{N^{\bar{A}}} \md{\rho^{i-\de+1} \nabla^i\left( \nabla \al \diamond u\right)}^p \rho^{-4}\dvol \\
&\le C\sum_{0 \le j \le i \le k} \int_{N^{\bar{A}}} \md{\rho^{j-\de}\nabla^j u}^p\md{\rho^{i-j+1}\nabla^{i-j+1} \al}^p \rho^{-4}\dvol \\
&\le C\left(\sum_{i = 0}^{k-1}\nm{\nabla^{i+1} \al \rho^{i+1}}^p_{C^0}\right) \nm{u}_{L^p_{k, \de, \bar{A}}}^p.
\eas
In the second line we used the bound on the norm of the $\diamond$-product.
\end{proof}

\begin{prop}
\label{4.2_D_t_inv_estimate}
 There is a constant $C_D$, independent of $\bar{A}$, such that for any  $u \in L^p_{k+1, \de, \bar{A}}(\nu(N^{\bar{A}}))$ which is $L^2_{\de\pm\eps}$-orthogonal to $\kappa_{\bar{A}}$ we have:
\e
\label{4.2_D_inv_estimate}
\nm{v}_{L^p_{k+1, \de, \bar{A}}} \le C_D \nm{D_{\bar{A}}v}_{L^p_{k, \de-1, \bar{A}}}
\e
\end{prop}
\begin{proof}
Write $u \in L^p_{k+1, \de, \bar{A}}(\nu(N^{\bar{A}}))$ as: 
\e
u = \al u + (1-\al)u \nonumber.
\e
Then clearly $\nm{u}_{L^p_{k+1, \de, \bar{A}}} \le \nm{\al u}_{L^p_{k+1, \de, \bar{A}}} + \nm{(1-\al)u}_{L^p_{k+1, \de, \bar{A}}}$. Let us consider the term $\nm{\al u}_{L^p_{k+1, \de, \bar{A}}}$ first. Note that $\al u$ is supported in $N^{\bar{A}}_{\AC}$, and that on the support of $\ka_{\AC, \bar{A}}$, the cut off function $\al$ is in fact equal to $1$. Thus $\al u$ is orthogonal to $\ka_{\bar{A}}$ by our orthogonality assumption on $u$. Using Proposition \ref{4.2_D_A_inv_estimate} we see that: 
\eas
\nm{\al u}_{L^p_{k+1, \de, \bar{A}}} &\le \tilde{C}_A\nm{D_{\bar{A}} (\al u) }_{L^p_{k, \de-1, \bar{A}} }
\eas
Now as $D_{\bar{A}}$ is a first-order operator whose coefficients depend pointwise on the $\Spin(7)$-structure (cf. \cite[Prop 3.4]{englebertDeformations}), we see that $D_{\bar{A}}(\al u) = \al D_{\bar{A}}u  + (\nabla \al) \diamond u$, where $\diamond$ is a family of bilinear products $\diamond:  T^*N^{\bar{A}} \ot E \ra E$ which is uniformly bounded in $t$. Thus we may apply Lemma \ref{4.2_bilinear_nabla_product_formula}  to see that in fact:
\eas
\nm{D_{\bar{A}}(\al u)}_{L^p_{k, \de-1, \bar{A}}} \le \tilde{C}_A\nm{\al D_{\bar{A}} u }_{L^p_{k, \de-1, \bar{A}}}+\frac{C_1\tilde{C}_A}{\log(t)}\nm{u}_{L^p_{k+1, \de, \bar{A}}}.
\eas 
In other words we have, if we also apply Lemma \ref{4.2_product_formula}:
\eas
\left(1-\frac{C_1\tilde{C}_A}{\log(t)}\right) \nm{\al u}_{L^p_{k+1, \de, \bar{A}}} &\le \tilde{C}_A\nm{\al D_{\bar{A}} u }_{L^p_{k, \de-1, \bar{A}}} \\
&\le \tilde{C}_AC_0\nm{D_{\bar{A}}u}_{L^p_{k, \de-1, \bar{A}}}.
\eas 
In particular, for $t$ sufficiently small, we get that:
\e
\label{4.2_inv_D_A}
\nm{\al u}_{L^p_{k+1, \de, \bar{A}}} \le \frac{C_D}{2}\nm{D_{\bar{A}}u}_{L^p_{k, \de-1, \bar{A}}}.
\e
We now note that the auxiliary Lemmas \ref{4.2_product_formula} and \ref{4.2_bilinear_nabla_product_formula} can equally well be proven for $1-\al$. Furthermore, the analogue for $N^{\bar{A}}_{\CS}$ of Proposition \ref{4.2_D_A_inv_estimate} is true. To see this note that the $L^p_{r, \si, \bar{A}}$ norms on $N^{\bar{A}}$ agree with the $L^p_{r, \si}$ norm for sections supported in $N^{\bar{A}}_{\CS}$. Furthermore, since $N^{\bar{A}}_{\CS}$ is already Cayley, $D_{\bar{A}}|_{N^{\bar{A}}_{\CS}} = D_{\CS}|_{N^{\bar{A}}_{\CS}}$, and so the result follows from Proposition \eqref{2_3_cs_bound_k}. We can therefore prove: 
\e
\label{4.2_inv_D_S}
\nm{(1-\al) u}_{L^p_{k+1, \de, \bar{A}}} \le \frac{C_D}{2}\nm{D_{\bar{A}}u}_{L^p_{k, \de-1, \bar{A}}}.
\e
Equations \eqref{4.2_inv_D_A} and \eqref{4.2_inv_D_S} taken together now give us: 
\eas
\nm{u}_{L^p_{k+1, \de, \bar{A}}} &\le \nm{\al u}_{L^p_{k+1, \de, \bar{A}}}+  \nm{(1-\al) u}_{L^p_{k+1, \de, \bar{A}}}\\
&\le C_D\nm{D_{\bar{A}}u}_{L^p_{k, \de-1, \bar{A}}}.
\eas
\end{proof}

\subsubsection*{Quadratic estimates}

We conclude this section on estimates by proving the quadratic estimates, which are consequences of the estimates in the compact and conically singular setting. 

\begin{prop}
Let $\de > 1$, $p > 4$ and $k \ge 1$. There are constants $E_Q>0$  and $C_Q > 0$, independent of $\bar{A}$ and the $\Spin(7)$-structure, such that for sufficiently small global scale $t > 0$, $s \in \cS$ sufficiently close to our initial $\Spin(7)$-structure $s_0$ and $v, w\in L^p_{k+1, \de, \bar{A}} $ with $\nm{v}_{L^{p}_{k+1, \de, \bar{A}}}, \nm{w}_{L^{p}_{k+1, \de, \bar{A}}} < E_Q$ we have:
\e
\label{4.2_Q_diff_estimate}
\nm{Q_{\bar{A}}(v, s)-Q_{\bar{A}}(w, s)}_{L^p_{k, \de-1, \bar{A}}} \le C_Q \nm{v-w}_{L^p_{k+1, \de, \bar{A}}}(\nm{v}_{L^p_{
k+1, \de, t}}+\nm{w}_{L^p_{k+1, \de, \bar{A}}}).
\e
\end{prop}
\begin{proof}
Let $u, v \in L^p_{k+1,\de, \bar{A}}(\nu_\ep(N))$ be given. By the Sobolev embedding Theorem \ref{2_3_Sobolev_embedding_acyl_ac}  for weighted spaces we see that there are embeddings $L^p_{k+1, \de, \bar{A}} \hookra C^k_{\de, \bar{A}}$. Here the Sobolev constants are bounded independent from $\bar{A}$ as it is invariant under rescaling of the $\AC$ pieces. Thus we have that $u$ and $v$ are $C^1$ and that their $C^1_{\de, \bar{A}}$-norms are bounded by $C \cdot  E_Q$. In particular we thus have that $\md{v}, \md{\nabla v} <  C\cdot  E_Q$ independently of $\bar{A}$. Hence we can invoke \cite[Lemma 4.3]{englebertDeformations} to obtain a pointwise bound of the form: 
\ea
 \md{Q_{\bar{A}}(v, s)-Q_{\bar{A}}(w, s)}_{C^{k+1}} \le& C(1+ \md{TN^{\bar{A}}}_{C^{k+1}}) \bigg(\md{v-w}_{C^{k+1}}(\md{v}_{{C^k}}+ \md{w}_{{C^k}}) \nonumber \\
 +& \md{v-w}_{C^{k}}(\md{v}_{{C^{k+1}}}+ \md{w}_{{C^{k+1}}}) \bigg).
\ea
In a similar fashion to how we prove the initial error estimates on $F_{\bar{A}}$, we can also show that we have: 
\eas
\md{TN^{\bar{A}}}_{C^{k+1}_{1, \bar{A}}} \le t^{\al_\de}, 
\eas
where $\al> 0$ is a positive constant, independent of $\bar{A}$. Thus the same reasoning as in the conically singular case \cite[Prop 4.30]{englebertDeformations} gives us the desired bound, with the constant being independent of $\bar{A}$.

\end{proof}

We now show that if all the pieces involved in the gluing are unobstructed, then the same is true for the glued manifolds. 
\begin{prop}
\label{4.3_unobstructed_def_op}
Let $4 < p < \infty$ and $k \ge 1$. Assume that both the $A_i$ and $N$ are unobstructed as Cayley manifolds at rate $\la$ and $\bar{\mu}$ respectively. Assume that $[\la, 1) \cap D_i = \emptyset$, $(1, \mu_i] \cap D_i = \emptyset$ and that all the cones which are glued in are unobstructed. Let $1 < \de < \mu_i$ be fixed. We then have that for sufficiently small $t$ the linearised deformation operator $D_{\bar{A}}$ is surjective. In particular, for any $w\in L^p_{k, \de-1}(E_{\cay})$ there is a unique $v \in \kappa_{\bar{A}} ^\perp$ such that $D_{\bar{A}}v = w$.
\end{prop}
\begin{proof}
We have that the operators $D_{\AC}$ and $D_{\CS}$ are surjective as maps from $L^p_{k+1, \de} \ra L^p_{k, \de-1}$, as increasing/decreasing the rate in the $\AC$/$\CS$ case cannot introduce a cokernel by Theorem \ref{2_1_change_of_index}. In particular they admit bounded right-inverses $P_{\AC} $ and $P_{\CS}$ respectively, which map $L^p_{k, \de-1}$ into $L^p_{k+1, \de}$. We would first like to show that $D_{\bar{A}}$ is surjective for sufficiently small values of $t$.
\begin{claim}
If there is a bounded linear map $P_{\bar{A}}: L^p_{k, \de-1} \ra L^p_{k+1, \de}$ such that the operator norm of $\id -D_{\bar{A}}  P_{\bar{A}} $ satisfies $\nm{\id -D_{\bar{A}}  P_{\bar{A}} } < 1$, then $D_{\bar{A}}$ is surjective. 
\end{claim}
\begin{claimproof}
By the continuous functional calculus in Banach spaces, the operator $D_{\bar{A}} P_{\bar{A}} = \id -(\id -D_{\bar{A}} P_{\bar{A}})$ has the bounded inverse $\sum_{i=0}^\infty (\id-D_{\bar{A}}P_{\bar{A}})^i$, as this sum converges by the assumption on the operator norm of $\id-D_{\bar{A}}P_{\bar{A}}$. Thus in particular $D_{\bar{A}}$ is surjective. 
\end{claimproof}

We now construct such a $P_{\bar{A}}$ by joining together $P_{\AC}$ and $P_{\CS}$, seen as operators on $N^{\bar{A}}_{\AC}$ and $N^{\bar{A}}_{\CS}$ respectively. Note that $P_{\CS}$ takes sections on $N$ to sections on $N$. Thus in particular, if $s \in L^p_{k, \de-1, \bar{A}}(E_{\cay})$ is a section on all of $N^{\bar{A}}$, then  $(1-\al)P_{\CS}((1-\al)s)$ defines a well-defined section which is supported on $N^{\bar{A}}_{\CS}$. Similarly, we have an identification of sections on $N^{\bar{A}}_{\AC}$ with sections on $A_i$ via the map $\Psi^i_{\bar{A}}$. This allows us to define the operator $\al
P_{\AC, \bar{A}}$ on $N^{\bar{A}}$, which takes section supported in $N^{A_i}_{\AC}$ to sections on $A_i$, applies $P_{A_i}$, and transports them back to section on $N^{\bar{A}}_{\AC}$. It has the noticeable property that $D_{\bar{A}} P_{\AC, \bar{A}} = \id$. We thus define: 
\e
P_{\bar{A}}(s) = (1-\al)P_{\CS}((1-\al)s) + \al P_{\AC, \bar{A}} (\al s).
\e
When precomposed with $D_{\bar{A}}$, we obtain:
\e
D_{\bar{A}}P_{\bar{A}}(s) -s =   (2\al(1-\al))s + \nabla (1-\al) \diamond_1 P_{\CS}((1-\al)s) + \nabla \al \diamond_2 P_{\AC, \bar{A}} (\al s),
\e
where $\diamond_1, \diamond_2$ are two bilinear products. Notice that $2\al(1-\al) \le \ha (\al + 1- \al)^2 = \ha$, thus in order to to prove the proposition we need to find $0 < K < \ha$ such that:
\eas
&\nm{\nabla (1-\al) \diamond_1 P_{\CS}((1-\al)s) }_{L^p_{k, \de-1, t}} \le \frac{K}{2} \nm{s}_{L^p_{k, \de-1, \bar{A}}}, \text{ and }\\
&\nm{\nabla \al \diamond_2 P_{\AC, \bar{A}} (\al s) }_{L^p_{k, \de-1, t}} \le \frac{K}{2} \nm{s}_{L^p_{k, \de-1, \bar{A}}}.
\eas

Let us consider the second inequality for concreteness. Proposition \ref{4.2_bilinear_nabla_product_formula} and the uniform boundedness of $P_{\AC, \bar{A}}$ allow us to write: 
\eas
\nm{\nabla \al \diamond_2 P_{\AC, \bar{A}} (\al s) }_{L^p_{k, \de-1, \bar{A}}} &\le \frac{C}{\log (t)}\nm{P_{\AC, \bar{A}} (\al s) }_{L^p_{k+1, \de, \bar{A}}} \\
&\le \frac{C}{\log (t)}\nm{P_{\AC, \bar{A}} (\al s) }_{L^p_{k+1, \de, \bar{A}}} \\
& \le \frac{C \tilde{C}_A}{\log (t)}\nm{\al s}_{L^p_{k, \de-1, \bar{A}}} \\
& \le \frac{C_1 \tilde{C}_A}{\log (t)}\nm{s}_{L^p_{k, \de-1, \bar{A}}}.
\eas
In the last line we applied Proposition \ref{4.2_product_formula}. Now note that for $t$ sufficiently small we can arrange that $ \frac{C_1 \tilde{C}_A}{\log (t)} < \frac{1}{4}$. The same reasoning applies to $P_{\CS}$, which concludes the proof.
\end{proof}

All the cones that appear in the gluing construction are unobstructed by assumption, thus the indices of the deformation operator of the pieces add up to the index of $D_{\bar{A}}$, as we have seen in  \cite[Remark 4.35]{englebertDeformations}. From this and the preservation of unobstructedness for the glued manifold by the previous lemma, we can conclude that $\ka_{\bar{A}} \simeq \ker D_{\bar{A}}$. 

\subsection{Finding a nearby Cayley}

We are now in good shape to prove the main theorem \ref{1_1_Main_theorem}.

\begin{thm}[Main theorem]
Let $(M, \Phi)$ be an almost $\Spin(7)$ manifold and $N$ a $\CS_{\bar{\mu}}$-Cayley in $(M, \Phi)$ with singular points $\{z_i\}_{i = 0, \dots, l}$ and rates $1 < \mu_i < 2$, modelled on the cones $C_i = \R_+ \times L_i \subset \R^8$ . Assume that $N$ is unobstructed in $\cM_{\CS}^{\bar{\mu}} (N, \{\Phi\})$. For a fixed $k\le l$, assume for each $i \le k$ that the $L_i$ are unobstructed as associatives (i.e. that the $C_i$ are unobstructed cones), and that $\cD_{L_i} \cap (1, \mu_i] = \emptyset$. For $1 \le i \le k$, suppose that $A_i$ is an unobstructed $\AC_\la$-Cayley with $\la < 1$, such that $\cD_{L_i} \cap [\la, 1) = \emptyset$. Let $\{\Phi_s\}_{s \in \cS}$ be a smooth family of deformations of $\Phi = \Phi_{s_0}$. Then there are open neighbourhoods $U_i$ of $C_i \in  \overline{\cM}_{\AC}^{\la}(A_i)$, an open neighbourhood $s_0 \in U \subset \cS$  and a continuous map:
\e
\Ga: U \times \cM_{\CS, \text{link}}^{\bar{\mu}} (N, \{\Phi\}) \times \prod_{i=1}^k U_i \longra \bigcup_{I \subset \{1, \dots, k\}} \cM_{\CS}^{\bar{\mu}_I}(N_I , \cS).
\e
Here we denote by ${\bar{\mu}_I}$ the subsequence, where we removed the $i$-th element for $i\in I$ from ${\bar{\mu}}$. Moreover, $N_I$ denotes the isotopy class of the manifold obtained after desingularising the points $z_i$ for $i \in I$ by a connected sum with $A_i$. 

This map is a local diffeomorphism of stratified manifolds. Thus away from the cones in $\overline{\cM}_{\AC}^{\la}(A_i)$ it is a local diffeomorphism onto the non-singular Cayley submanifolds. It maps the point $(s, \tilde{N}, \tilde{A}_1, \dots, \tilde{A}_k)$  into $\cM_{\CS}^{\bar{\mu}_I}(N_I , \cS)$, where $I$ is the collection of indices for which  $\tilde{A}_i = C_i$. This corresponds to partial desingularisation.
\end{thm}

\begin{proof}
Let $k \ge 1$ and $p > 4$ be fixed. We first find a solution to the equation $F_{\bar{A}}(v) = 0$ for a fixed $\Spin(7)$-structure via an iteration scheme. For this, fix an $\si > 0$ such that $\si < \frac{C_Q}{2}$. We will construct sections $v_i^{\bar{A}} \in L^p_{k+1, \de, \bar{A}}$ with $i \in \N$ which satisfy:
\ea
D_{\bar{A}}v_{i+1}^{\bar{A}} = -F_{\bar{A}}(0) - Q_{\bar{A}}(v_i^{\bar{A}}),  \nonumber \\
\label{4.3_sequence}
v_i^{\bar{A}} \perp^{L^2_{\de\pm\eps}} \ka_{\bar{A}} \text{ and } \nm{v_i^{\bar{A}}}_{L^p_{k+1, \de, {\bar{A}}}} < \si.
\ea
For this, define first $v_0^{\bar{A}} = 0$ for any ${\bar{A}}$ with sufficiently small $t$ and $\de > 1$. Then Proposition \ref{4.3_unobstructed_def_op} allows us to find a unique pre-image $v_i^{\bar{A}}$ of $-F_{\bar{A}}(0)-Q_{\bar{A}}(v_0^{\bar{A}}) = -F_{\bar{A}}(0)$. From our estimate on the inverse of $D_{\bar{A}}$ on sections which are orthogonal to the approximate kernel $\ka_{\bar{A}}$ from Proposition \ref{4.2_D_t_inv_estimate} we see that: 
\eas
\nm{v_1^{\bar{A}}}_{L^p_{k+1, \de, {\bar{A}}}} &\le C_D\nm{D_{\bar{A}} v_1^{\bar{A}}}_{L^p_{k, \de-1, {\bar{A}}}} \\
&\le C_D\nm{F_{\bar{A}}(0)}_{L^p_{k, \de-1, {\bar{A}}}} \\
&\le C_D C_F t^{\nu(\mu-\de)}.
\eas
Here we used Proposition \ref{4.2_tau_estimate} on our initial error estimate in the last line. We see that for sufficiently small $1<\de<\mu$  the initial error will become arbitrarily small. Thus for $t_0 > 0$ sufficiently small we have:
\eas
\nm{v_1^{\bar{A}}}_{L^p_{k+1, \de+1, {\bar{A}}}} \le C_D C_F t^{\nu(\mu-\de)} < \frac{\si}{4},
\eas
for all $t\in (0, t_0]$. Suppose now that we have constructed $v_i^{\bar{A}}$ for some $i \in \N$, such that $\nm{v_1^{\bar{A}}}_{L^p_{k+1, \de,{\bar{A}}}}< \si$. We then find the pre-image $v_{i+1}^{\bar{A}}$ of $-F_{\bar{A}}(0)-Q_{\bar{A}}(v_i^{\bar{A}})$ and use our estimate on $Q_{\bar{A}}$ from Proposition \ref{4.2_Q_diff_estimate} to show the following: 
\eas
\nm{v_{i+1}^{\bar{A}}}_{L^p_{k+1, \de, {\bar{A}}}} &\le C_D\nm{D_{\bar{A}} v_i^{\bar{A}}}_{L^p_{k, \de-1, {\bar{A}}}} \\
&\le C_D(\nm{F_{\bar{A}}(0)}_{L^p_{k, \de-1, {\bar{A}}}}+\nm{Q_{\bar{A}}(v_i^{\bar{A}})}_{L^p_{k, \de-1, {\bar{A}}}}) \\
&\le C_D C_F t^{\nu(\mu-\de)}  + C_Q \nm{v_{i}^{\bar{A}}}^2_{L^p_{k+1, \de+1, {\bar{A}}}} \\
&\le \frac{\si}{4} + C_Q \si^2 < \si.
\eas
We can now iterate this procedure to obtain a sequence $\{v_i^{\bar{A}}\}_{i \in \N}$ for every $t \in (0, t_0]$ which satisfies our requirements \eqref{4.3_sequence}. Note that we are free to choose $0< \si < \frac{C_Q}{2}$. This family of sequences converges uniformly in ${\bar{A}}$,as we have the bounds: 
\eas
\nm{v_{i+1}^{\bar{A}}-v_i^{\bar{A}}}_{L^p_{k+1, \de, {\bar{A}}}} &\le C_D \nm{Q_{\bar{A}}(v_{i}^{\bar{A}})-Q_{\bar{A}}(v_{i-i}^{\bar{A}})}_{L^p_{k, \de-1, {\bar{A}}}} \\
&\le C_D C_Q (\nm{v_{i}^{\bar{A}}}_{L^p_{k+1, \de, {\bar{A}}}}+\nm{v_{i-1}^{\bar{A}}}_{L^p_{k+1, \de, {\bar{A}}}})\nm{v_{i}^{\bar{A}}-v_{i-1}^{\bar{A}}}_{L^p_{k+1, \de, {\bar{A}}}} \\
& \le 2C_D C_Q \si  \nm{v_{i}^{\bar{A}}-v_{i-1}^{\bar{A}}}_{L^p_{k+1, \de, {\bar{A}}}}.
\eas
If we choose $\si$ small enough, we can ensure that:
\eas
\nm{v_{i+1}^{\bar{A}}-v_i^{\bar{A}}}_{L^p_{k+1, \de,{\bar{A}}}}  < \ha \nm{v_{i}^{\bar{A}}-v_{i-1}^{\bar{A}}}_{L^p_{k+1, \de, {\bar{A}}}}.
\eas
 Thus $\{v_i^{\bar{A}}\}_{i \in \N}$ is a Cauchy sequence in $L^p_{k+1,\de, {\bar{A}}}(\nu(N^{\bar{A}}))$ for each $\bar{A}$ simultaneously. We can thus find limits $v_\infty^{\bar{A}} \in L^p_{k+1, \de, {\bar{A}}}(\nu(N^{\bar{A}}))$. Since both $D_{\bar{A}}$ and $Q_{\bar{A}}$ are continuous maps of Banach manifolds, we have: 
\eas
D_{\bar{A}} v_\infty^{\bar{A}} &= \lim_{i \ra \infty} D_{\bar{A}} v_{i+1}^{\bar{A}} \\
&= \lim_{i \ra \infty} -F_{\bar{A}}(0) - Q_{\bar{A}}(v_i^{\bar{A}}) \\
&= -F_{\bar{A}}(0) - Q_{\bar{A}}(v_\infty^{\bar{A}}).
\eas 
Thus $F_{\bar{A}}(v_\infty^{\bar{A}}) = 0$. We then immediately get smoothness for $v_\infty^{\bar{A}}$ by Proposition \ref{3_1_F_A_smooth}.  By Proposition \ref{2_1_deformation_op_prop} we can conclude that $\tilde{N}^{\bar{A}} = \exp_{v_\infty^{\bar{A}}} (N^{\bar{A}})$ is a family of smooth Cayley submanifolds, as the family clearly only varies in a compact subset of $M$. The manifold $\tilde{N}^{\bar{A}}$ has the same topological type as $N^{\bar{A}}$ and together the $\tilde{N}^{\bar{A}}$
 form the desired desingularisation. 
 
Thus we can define a map $\Ga$ as above on the slice $\{\Phi\} \times\cM_{\CS}^{\bar{\mu}} (N, \{\Phi\}) \times \prod_{i=1}^k U_i$. We would now like to extend this map when the ambient $\Spin(7)$-structure is allowed to vary. For this, we first choose a trivialisation $T:\cS \times\cM_{\CS}^{\bar{\mu}}(N, \{\Phi\}) \simeq \times\cM_{\CS}^{\bar{\mu}} (N, \cS)$, which can be done by unobstructedness of $N$, using Theorem \ref{2_3_cs_structure}. Now we can repeat the above iteration scheme for $\Phi'$, where we now glue $\bar{A}$ onto $N' = T(\Phi', N)$. From this we see that smoothly varying the $\Spin(7)$-structure leads to a smooth change in the resulting submanifold.
 
 Note that $\nm{v_{\infty}^{\bar{A}}}_{L^p_{k+1, \de, {\bar{A}}}} \le 2 \nm{v_{0}^{\bar{A}}}_{L^p_{k+1, \de+1, {\bar{A}}}} \le Ct^{\nu(\mu-\de)}$, and thus as the scale $t$ tends to $0$, the resulting Cayley will converge in $L^p_{k+1,\de,\bar{A}}$ (thus in $C^{k}_{loc}$) to $N^{\bar{A}}$, which in turn converges in the sense of currents to the conically singular $N$. As we also have $C^k_{loc}$ convergence for any $k \ge 1$, we get $C^\infty_{loc}$ convergence as well. Moreover, there is nothing special about reducing the global scale as opposed to reducing only a subset of the scales to $0$. In this case the same argument localised to the singular points in question gives the $C^\infty_{loc}$ convergence to the partially desingularised $N$.
 
Finally, this construction is smooth in the gluing pieces away from cones. Indeed, varying the pieces gives rise to a smooth change of the p.d.e. $F_{\bar{A}}(v) = 0$, and all the constants involved in the iteration scheme remain valid. Thus the result will also vary smoothly.
\end{proof}

We would like to point out that Theorem \ref{1_1_Main_theorem} is not the only possible gluing result in this setting. What is needed in the construction are the following three ingredients. Whenever these are true, we can prove a corresponding gluing result.
\begin{itemize}
\item The initial error $\nm{F_t(0)}_{L^p_{k+1, \de, t}}$ needs to go to zero as the global neck size $t \ra 0$. 
\item The quadratic estimate \eqref{4.2_Q_diff_estimate} needs to hold for some constant $C_Q$.
\item The linearised operator needs to be invertible orthogonal to its kernel, and has to have uniformly bounded norm.
\end{itemize}

The first two items above are true as long as our initial approximation gets better in a $C^1$ sense as $t \ra 0$, and we know how to handle the local model of the noncompact piece (in this case a cone). In particular we do not need unobstructedness of the $\AC$ and $\CS$ pieces for these two items. We do however need it for the last item, where it is crucial that the glued operator is surjective and has a well understood behaviour in the $L^p_{k, \de, t}$ norms as $t \ra 0$. In the Theorem \ref{1_1_Main_theorem} we chose the rates of both pieces to be near $1$, and then included the slighty tricky rate $1$ into the moduli space of $\CS$ Cayleys. However, provided that $A_i \in \cM_{\AC}^{-\eps}$ and  $\cM_{\CS}^{1+\eps}$ are unobstructed (where we now allow the points to move in the $\CS$ moduli space), we can define a gluing map:

\e
\label{3_3_gluing_around_0}
\tilde{\Ga}: U \times \cM_{\CS}^{\bar{\mu}} (N, \{\Phi\}) \times \prod_{i=1}^k U_i  \longra \bigcup_{I \subset \{1, \dots, k\}} \cM_{\CS}^{\bar{\mu}_I}(N_I , \cS).
\e
Here $U_i \subset \overline{\cM}_{\AC}^{-\eps}$ are now not including the translations. They are included in the $\CS$ moduli space. Essentially we can define a gluing map whenever we have rates $\la < 1 < \mu$ for which the pieces are unobstructed, and we can include the translations and rotations manually on the conically singular side. 

Note however that if we are missing some critical rates, in the sense that there is a critical rate $\de \in \cD$ which not accounted for on either the $\AC$ or the $\CS$ piece, then the gluing map will not be surjective. So for instance, if we are given a cone with no critical rates in the range $(0, 1)$, we still have surjectivity of the map $\tilde{\Ga}$.

\subsection{Desingularisation of immersed Cayley submanifolds}

Note that if two Cayley planes intersect positively, then they cannot be desingularised by a minimal surface. This is due to the angle criterion, proved by Lawlor in \cite{lawlorAngleCriterion1989}. In the language of Lawlor, the conditions that the planes be Cayley and intersect positively imply that the characteristic angles between the two planes satisfy $\Th_1 + \Th_2 +\Th_3 +\Th_4 = 2\pi > \pi$, and so they do not satisfy the angle criterion. 
As an immediate consequence of Lemma \cite[Lemma 4.25]{englebertDeformations} and Theorem \ref{1_1_Main_theorem} we obtain the following desingularisation result, which is optimal by the preceding discussion.

\begin{thm}[Desingularisation of immersions]
\label{4_4_desingularisation}
Let $N$ be an unobstructed immersed Cayley submanifold which admits a negative self-intersection at $p \in N$. Then there is a family of Cayley submanifolds  with one fewer singular point $\{N_t\}_{t\in (0,\eps)}$ such that $N_t \ra N$ in the sense of currents and also in $C^\infty_{loc}$ away from the singularity as $t\ra 0$.
\end{thm}
\begin{ex}
Consider the $\Spin(7)$-manifold $(T^8, \Phi_0)$, which is obtained as a quotient of $(\R^8, \Phi_0)$  by the lattice of integer points. Consider any affine plane in $\R^8$ which descends to a closed manifold in the quotient. Take for instance the special Lagrangian plane $\R^4 \subset \C^4$. It admits a $16$-dimensional space of Cayley deformations, however a $12$-dimensional subset of these is generated by rotations and thus not preserved in the quotient (as the image will be of a different topological type). What remains are the $4$-dimensional family of translation, which descend to the obvious translations of a $T^4 \times \{0\} \subset T^8$. Notice however that its Cayley moduli space has dimension $\ha(\sigma(T^4) - \chi(T^4)) = 0$ by  \cite[Example 4.12]{englebertDeformations}. Thus this four-torus is obstructed as a Cayley in the moduli space $\cM(T^8, \Phi)$. We can however modify the $\Spin(7)$-structure near $T^4$ so that the submanifold becomes unobstructed in the new moduli space $\cM(T^8, \Phi)$. In particular if we take the union of a finite number of such tori that each intersect each other negatively, we can construct a $\Spin(7)$ structure in which we can desingularise the union of tori using our gluing theorem \ref{1_1_Main_theorem} to obtain a connected sum of tori in a $(T^8, \Phi)$, where $\Phi$ is a small perturbation of the usual flat structure $\Phi_0$.
\end{ex}
\begin{ex}
Consider the CY fourfold $M = \{z_0^6 +z_1^6+z_2^6+z_3^6+z_4^6+z_5^6 = 0 \} \subset \C P^5$. In this manifold we can construct special Lagrangian and complex submanifolds which intersect in a point. The complex surface is $N = \{z_1 = iz_2, z_3 = iz_4\}$. For the special Lagrangian we choose the fixed-point locus of the following anti-holomorphic involution:
\eas
\si([z_0, z_1,z_2,z_3,z_4,z_5])=[\bar{z}_0,\bar{z}_1,\bar{z}_2,\bar{z}_3,e^{i\frac{\pi}{3}}\bar{z}_5].
\eas
We have that $L = \Fix (\si)$  is a special Lagrangian submanifold by \cite[Prop. 12.5.2]{joyceRiemannianHolonomyGroups2007}. They intersect negatively, however it turns out that the special Lagrangian is obstructed. Thus as in the previous example, we can only say that there is a Cayley in a nearby $\Spin(7)$-structure. More generally, special Lagrangians tend to be obstructed, as we see from Example 4.12 in \cite{englebertDeformations}. There we show that the obstruction space of a special Lagrangian $L$ in a CY fourfold $M$ is given by:
\eas
\cO(L) \simeq \Ho^0(L) \op \Ho^{2,-}(L).
\eas
In particular, if $L$ is connected we then have $\dim \cO(L) = 1 + b^{2,-}$. Not at least one obstruction seems to stem from possible choices of parameters in the Cayley form in the torsion free setting:
\eas
\Phi_\phi = \Re(e^{i\phi}\Om) + \ha \om \wedge \om. 
\eas
Here any choice of $\phi \in \R$ and any choice of $\om$ in the Kähler cone of $(M, \om)$ gives rise to a valid Cayley form. However note that if $L$ is special Lagrangian in $M$, i.e. $\Re(\Om)|_L= \dvol_L$, then the moduli space $\cM(L, \Phi_\phi)$ with $\phi \ne 0$ is necessarily empty, for by Stokes' theorem whenever $\tilde{L}$ is homologous to $L$:
\eas
\int_{\tilde{L}} \Phi_\phi= \int_{L} \Phi_\phi = \int_L \Re(e^{i\phi}\Om)  < \int_L \Re(\Om) = \vol(L), 
\eas
And thus no calibrated submanifolds in the homology class of $L$ can exist for $\Phi_\phi$, in the torsion free setting. We can remove the obstructions associated to $\phi$ manually by quotienting $M$ by an antiholomorphic involution. The only $\Spin(7)$ structures that descend to the quotient satisfy $\phi = 0$. However we do not know of any way to remove the obstruction coming from the choice of $\om$. We have yet to find any truly interesting examples of a desingularisation of a union of a complex surface and a special Lagrangian. 
\end{ex}

\addcontentsline{toc}{section}{References}
\bibliographystyle{alpha}

\bibliography{bibliography}

\begin{thebibliography}{{Eng}23}

\bibitem[{Eng}23]{englebertDeformations}
Gilles {Englebert}.
\newblock {Conically Singular {{Cayley}} Submanifolds {{I}}: {{Deformations}},
  preprint}.
\newblock \url{https://doi.org/10.48550/arXiv.2309.07830}, 2023.

\bibitem[HL82]{HarvLaws}
Reese Harvey and H.~Blaine Lawson.
\newblock Calibrated geometries.
\newblock {\em Acta Mathematica}, 148:47--157, 1982.

\bibitem[Joy04a]{joyceReg}
Dominic~D. Joyce.
\newblock Special {{Lagrangian}} submanifolds with isolated conical
  singularities. {{I}}. {{Regularity}}.
\newblock {\em Annals of Global Analysis and Geometry}, 25:201--251, 2004.

\bibitem[Joy04b]{joyceDesingularisation}
Dominic~D. Joyce.
\newblock Special {{Lagrangian}} submanifolds with isolated conical
  singularities. {{III}}. {{Desingularization}}, the unobstructed case.
\newblock {\em Annals of Global Analysis and Geometry}, 26:1--58, 2004.

\bibitem[Joy07]{joyceRiemannianHolonomyGroups2007}
Dominic~D. Joyce.
\newblock {\em Riemannian Holonomy Groups and Calibrated Geometry}.
\newblock Oxford {{Graduate Texts}} in {{Mathematics}}. {Oxford University
  Press}, 2007.

\bibitem[Kov05]{KovalevFibration}
Alexei Kovalev.
\newblock Coassociative {{K3}} fibrations of compact {$G_2$}-manifolds, preprint.
\newblock \url{ https://doi.org/10.48550/arXiv.math/0511150 }, 2005.

\bibitem[Law89]{lawlorAngleCriterion1989}
Gary~R. Lawlor.
\newblock The angle criterion.
\newblock {\em Inventiones mathematicae}, 95:437--446, 1989.

\bibitem[Loc87]{lockhartFredholmHodgeLiouville1987}
Robert~B. Lockhart.
\newblock Fredholm, {{Hodge}} and {{Liouville Theorems}} on {{Noncompact
  Manifolds}}.
\newblock {\em Transactions of the American Mathematical Society}, 301:1--35,
  1987.

\bibitem[Lot08]{lotayDesingularizationCoassociative4folds2008}
Jason~D. Lotay.
\newblock Desingularization of {{Coassociative}} 4-folds with {{Conical
  Singularities}}.
\newblock {\em Geometric and Functional Analysis}, 18:2055--2100, 2008.

\bibitem[McL98]{mcleanDeformationsCalibratedSubmanifolds1998}
Robert~C. McLean.
\newblock Deformations of calibrated submanifolds.
\newblock {\em Communications in Analysis and Geometry}, 4:705--747, 1998.

\bibitem[Moor17]{mooreDeformationTheoryCayley2017}
Kim Moore.
\newblock {\em Deformation Theory of {{Cayley}} Submanifolds}.
\newblock PhD thesis, University of Cambridge, 2017.

\end{thebibliography}

\medskip

\noindent{\small\sc The Mathematical Institute, Radcliffe
Observatory Quarter, Woodstock Road, Oxford, OX2 6GG, U.K.

\noindent E-mail: {\tt gilles.englebert@maths.ox.ac.uk.}}

\end{document}